\documentclass[12pt]{amsart}
\usepackage{color}

\def\N{{\mathbb N}}
\def\M{{\mathbb M}}
\def\R{{\mathbb R}}

\def\be{\begin{enumerate}}
\def\ee{\end{enumerate}}

\newcommand{\1}{{ 1\hspace*{-0.45ex}\rule{0.05ex}%
       {1.5ex}\hspace*{0.45ex}}}   % a poor man's bold version of 1

\newcommand{\BM}{Brownian motion}

\def\bc{\begin{center}}
\def\beq{\begin{equation}}
\def\bea{\begin{eqnarray}}
\def\beas{\begin{eqnarray*}}
\def\eeq{\end{equation}}
\def\eea{\end{eqnarray}}
\def\eeas{\end{eqnarray*}}
\def\ec{\end{center}}      
\def\f{\frac}

\numberwithin{equation}{section}

\newtheorem{theorem}{Theorem}[section]

\newtheorem{lemma}{Lemma}[section] 

\newtheorem{example}{Example}[section]
\newtheorem{remark}{Remark}[section]
\def\CP{ {\mathcal P}}
\def\CF{ {\mathcal F}}
\def\CN{ {\mathcal N}}

\begin{document}
\author{Siva R. Athreya}

\address{Siva R. Athreya \\ 8th, Mile Mysore Road,\\ Indian Statistical Institute
  \\Bangalore 560059 India.}
\email{athreya@isibang.ac.in}

\author{Thomas S. Salisbury} 

\address{Thomas S. Salisbury\\ Dept. of Math \& Stat.\\ York University\\ Toronto, Ontario\\
Canada M3J 1P3}  \email{salt@yorku.ca}


\title{Blowup and  Conditionings of $\psi$-super Brownian Exit Measures}

\date{\today}

\begin{abstract}
We extend earlier results on conditioning of super-Brownian motion to general branching rules. We obtain representations of the conditioned process, both as an $h$-transform, and as an unconditioned superprocess with immigration along a branching tree. Unlike the finite-variance branching setting, these trees are no longer binary, and strictly positive mass can be created at branch points. This construction is singular in the case of stable branching. We analyze this singularity first by approaching the stable branching function via analytic approximations. In this context the singularity of the stable case can be attributed to blowup of the mass created at the first branch of the tree. Other ways of approaching the stable case yield a branching tree that is different in law. To explain this anomaly we construct a family of martingales whose backbones have multiple limit laws. 
\end{abstract}
\maketitle

\section{Introduction}\label{intro}

The $\psi$- super-Brownian motion $X_{t}$ is a measure valued diffusion with branching mechanism given by 
\[ \psi(\lambda) = a_1 \lambda + 
a_2 \lambda^2 + \int_0^\infty [e^{-\lambda r} -1 + \lambda r]\, \pi(dr),
\]
where $\lambda \geq 0$, $a_{1} \in \R$, $a_2\geq 0$, and $\pi(\cdot)$ is the associated L\'evy measure.   More precisely, the log-laplace functional of $X_{t}$,
\[ u(t,x) = - \log E_{\delta_{x}} \exp(- \int \phi(y) dX_{t}(y) )\]
solves the initial value problem 
\[ \frac{\partial u }{\partial t}  = \frac{1}{2} \Delta u - \psi (u), \,\, u(0,\cdot) = \phi(\cdot),\]
whenever $\phi$ is a non-negative bounded continuous function. These diffusions can be realised as scaling limits of a system of particles that perform  branching Brownian motions. The exit measure $X^{D}$ of such a process from a bounded domain $D$ is a random  measure on the boundary of $D$, supported on the set of points where the particles first exit the domain.  It is known that for $d \geq 3$ that the random measure does not see points (see Le Gall \cite{legall1} and Dynkin \cite{dy0}).

In Salisbury and Verzani \cite{sv1} and \cite{sv2}  the exit measure of super Brownian motion,$X^{D}$ with critical binary branching (that is, with branching mechanism $\psi(\lambda)=2\lambda^2$) is conditioned to charge small balls $\Delta^{z_i}_\epsilon=B(z_i,\epsilon)\cap\partial D$, where $z_1,\dots,z_n\in\partial D$. Letting $\epsilon\to 0$ they obtain a conditioned process, which is a martingale transform of super Brownian motion by a ``polynomial'' martingale of degree $n$. They describe this process in terms of a ``backbone'' consisting of a binary tree that realizes the trajectory of the mass that reaches $z_1,\dots,z_n$.  These results do not generalize to a stable branching mechanism $\psi(\lambda)=c\lambda^{1+\beta}$, since
``polynomials'' of the exit measure will not even be integrable, let alone give martingales, when $n\ge 2$. Understanding why formed the primary motivation for this paper.

Our aim  is to understand conditioning based on general branching mechanisms  well enough to analyze how these martingales blow up as we approach the stable case.  Consider a bounded domain $D  \subset \R^{d} $, when $d \geq 3$. Let $\epsilon >0$ and fix  $z_{i} \in \partial D$ for $i =1 ,2, \ldots, n$. As above, define $\Delta^{z_i}_\epsilon$ to be a ball on $\partial D$ of radius $\epsilon$.  We condition $\psi$-super Brownian motion to hit balls $\Delta^{z_i}_\epsilon$ and obtain both the martingale transform that represents this process, and the probabilistic representation of this process in terms of immigrating mass along a branching ``backbone''. We explicitly describe the evolution of the backbone tree, the manner mass is generated along the backbone, and the way it evolves afterwards (See Theorem \ref{theorem} in Section \ref{BrParticleDescrip}).  Unlike the results of \cite{sv1} and \cite{sv2}, we now have to handle the probabilities of multiple branches and the distribution of positive mass created at the branch points of the tree. 

We first take the limit  as $\epsilon\to 0$, when $d \geq 4 $ and $\psi$ is a real-analytic function.  We establish that the limit exists and is a martingale change of measure (see Theorem \ref{limitingXtransform} and Theorem \ref{h1h2} in Section \ref{asymptotics} ).  The proof here follows the road map laid out in \cite{sv1} but the significant estimates appear to require more delicate arguments (Lemma \ref{thetaest} and Lemma \ref{bound}). The limiting backbone  produced is a tree with precisely $n$ leaves (see Section \ref{limitBackbone}), but is no longer binary. Next we  let the branching mechanism approach the stable case ($\psi(\lambda) = c\lambda^{1 + \beta}$).   With this particular order of  taking the two limits, the backbone remains well behaved, as does the mass creation at non-branch points of the backbone. However, the mass created at branch points gets arbitrarily large, resulting in an explosion that allows us to pinpoint the source of the blowup (See Section \ref{explosion} and Theorem \ref{supermart}).  We also consider other ways of approaching the stable case, for example, to simply let $\epsilon\to 0$ with branching mechanism $\psi(\lambda)=c\lambda^{1+\beta}$. In this case,  the backbone remains well-behaved but has a different law than in the previous mechanism. To explain this anomaly we construct a simpler family of martingales whose backbones have multiple limit laws, interpolating between analogues of both types obtained above (See Section \ref{otherlimits}).

The decomposition of the conditioned superprocesses in terms of an
``immortal backbone''  has been considered in the literature, in other contexts.   In particular, $h$-transforms of critical super-Brownian motion  have been studied by Roelly-Coppoletta and Roualt \cite{rr1}, Evans and Perkins \cite{ep}, and  Overbeck \cite{ov}, \cite{ov2}, while the immortal particle representation was originally discovered in Evans \cite{ev}.  Other studies include Serlet \cite{serlet} and Etheridge \cite{etheridge}. Salisbury and Sezer \cite{ssezer} and Verzani \cite{verzani} consider more general conditionings for binary branching.  Moras \cite{moras} considers specific classes of unbounded domains $D$, again for binary branching.

Versions for non-binary branching were studied in Etheridge and Williams 
\cite{ew} and  Kyprianou et al \cite{klmr} (we learned of the latter after completing our research). In \cite{ew}, the super Brownian motion on all of $\R^{d}$, with stable $(1 + \beta )$ branching, is conditioned on survival until some fixed time $T$. This backbone has a Poisson number of
immortal trees (conditioned on there being at least one), along which mass
(conditioned to die before time $T$ ) is immigrated. The rate of immigration
is random and there is additional immigration  whenever the immortal tree
branches.  In the limit as $T \rightarrow \infty$ , the immortal trees
degenerate to the Evans immortal particle and the immigration (of
unconditioned mass) along the particle is dictated by a stable
subordinator. In \cite{klmr}, to study the  travelling wave equation
associated to the parabolic semi-group equation of  $\psi$-super-Brownian
motion, the authors show a similar  immortal backbone  on all of $\R$ (which they call a
''spine decomposition'').

\subsection{Layout of the paper}
The rest of the paper is organised as follows.  In the next section we discuss preliminaries with regard to conditioned diffusions, $\psi$-super Brownian motion  and potentials. The Palm formula in Lemma \ref{palmformula} and identity for potentials in Lemma \ref{harmonic}, presented here,  are used significantly in the  proof of  main results.  In Section \ref{BrParticleDescrip} we  prove Theorem \ref{theorem}. In particular, we define a martingale change of measure representing various conditionings and for each provide a description of the associated branching backbone representation. 

In Section \ref{hittingNpoints}, we consider a specific condition namely that of  the exit measure charging finitely many points on the boundary. The work here (as explained earlier) is divided into three parts. First for $\psi$-analytic, we  condition the exit measure to hit balls of radius $\epsilon$ and   establish a limit as $\epsilon \rightarrow 0$. This is proved in Theorem \ref{h1h2} and Theorem \ref{limitingXtransform}. A limiting backbone is then described in Section \ref{limitBackbone}.  Secondly, in Section \ref{explosion} we consider the second limiting procedure of $\psi$-analytic to approximate $\psi(\lambda) = c\lambda^{1+\beta}$ for $0< \beta \leq 1$. We explain the explosion effect precisely.  Finally in Section \ref{otherlimits}, we begin by explaining other possible limits (if the order of limits done earlier are interchanged). We conclude  by  constructing a family of related martingales whose backbones have multiple limit laws. These interpolate between analogues of the two types of limits obtained earlier.
 
\subsection{Acknowledgements} Research done in this paper was supported in part by NSERC. During the completion of this work: the authors visited the University of British Columbia;  the Fields Institute; Siva Athreya visited York University in Toronto; and Tom Salisbury visited the Indian Statistical Institute, Bangalore. We would like to thank all  these places for their kind hospitality.

\section{Preliminaries}\label{preliminaries}

\subsection{Notation}\label{notation}
For each set $A$, let $|A|$ denote its cardinality, and let $\CP(A)$ denote the collection of partitions of $A$. Impose an order on $A$. Then for any $\sigma \in \CP(A)$, we may order the sets in $\sigma$ by their smallest elements. Letting $\sigma(j)$ denote the $j$th element of $\sigma$ in this
order, we may switch at will between the following two notations:

\[ \prod_{C \in \sigma} \langle X^D, v^C \rangle = \prod_{j=1}^{| \sigma |} \langle X^D, v^{\sigma(j)}\rangle . \]

The following elementary combinatorial result will be convenient
\begin{lemma} \label{settheory}
Let $A \subset B \subset C.$ be subsets of $\{1,2 \ldots, n\}$. Then 
\[ \sum_{A\subset B \subset C} (-1)^{| B | } = (-1)^{| C | }1_{A=C} \]
\end{lemma}
\begin{proof} Both sides equal $(-1)^{|A|}(1-1)^{|C\setminus A|}$. See Lemma 2.1 of \cite{sv1} 
\end{proof}

\subsection{ Facts about conditioned diffusions} \label{facd}

First we recall some familiar formulae for conditioned Brownian motion.

Let $B$ be $d$-dimensional \BM\/ started from $x$, under a probability
measure $P_x$. Write 
$\tau_D=\tau_D(B)$ for the first exit time of $B$ from $D$.
Let $g:D\to[0,\infty)$ be bounded on compact subsets of $D$, and set
\begin{equation*}
  L_g=\frac{1}{2}\Delta -g.
\end{equation*}
Let $\xi_t$ be a process which, under a probability law $P^g_x$, has the
law of a diffusion with generator $L_g$ started at $x$ and killed upon leaving 
$D$. In other words, $\xi$ is a \BM\/ in $D$, killed at rate $g$. Write $\zeta$ 
for the lifetime of $\xi$. Then
\begin{equation}
  \label{killed process}
  P^g_x(\xi_t \in A, \zeta > t) = 
  P_x\Big(\exp -\int_0^t \, g(B_s)  \,ds,B_t \in A, \tau_D > t \Big). 
\end{equation}
Let $U^g f(x) = \int_0^\infty P^g_x(f(\xi_t) \1_{\{\zeta > t \}}) \,dt$ be the
potential operator for $L_g$. If $g=0$ we write $U$ for $U^g$.
If $0 \leq u$ is $L_g$-superharmonic, then the law 
of the $u$-transform of $\xi$ is determined by the formula
\begin{equation*} 
  P^{g,u}_x(\Phi(\xi) \1_{\{\zeta > t \}} ) 
  = \frac{1}{u(x)} 
  P_x^g(\Phi(\xi) u(\xi_t) \1_{\{\zeta > t\}}) 
\end{equation*} 
for $\Phi(\xi) \in \sigma\{\xi_s; s \leq t\}$.
Assuming that $0 < u < \infty$ on $D$, this defines a diffusion on 
$D$. 
If $u$ is $L_g$-harmonic, 
then it dies only upon reaching $\partial D$. 
In fact, the generator of the $u$-transform is 
\begin{equation*}
L_{g,u}f=\frac{1}{u} L_g(uf) 
= \frac{1}{2}\Delta f +\frac{1}{u}\nabla u\cdot\nabla f.
\end{equation*}

If $u = U^g f$ for some $f \geq 0$
(that is, if $u$ is a potential) then the $u$-transform dies in the
interior of $D$, and 
$P^{g,u}_x$ satisfies 
$$
  P^{g,u}_x(\Phi(\xi)) = \frac{1}{u(x)} 
  \int^\infty_0 P_x^g(  
  \Phi(\xi_{\leq t}) f(\xi_t)\1_{\{\zeta > t\}}) \,dt,
$$
where $\xi_{\leq t}$ is the process $\xi$ killed at time $t$. See \cite{a}.

More generally, if $L_gu=-f$ then by breaking $u$ into a potential plus a 
harmonic function one has that
\begin{equation}
  \label{u process}
  P^{g,u}_x(\Phi(\xi)1_{\{\zeta<\tau_D\}}) = \frac{1}{u(x)} 
  \int^\infty_0 P_x^g(  
  \Phi(\xi_{\leq t}) f(\xi_t)\1_{\{\zeta > t\}}) \,dt.
\end{equation}

\subsection{$\psi$ super Brownian-motion}\label{psiSBMsection}

Let \begin{equation} \label{psi} \psi(\lambda) = a_1 \lambda + 
a_2 \lambda^2 + \int_0^\infty [e^{-\lambda r} -1 + \lambda r]\, \pi(dr),
\end{equation}
where $\lambda \geq 0, a_2\geq 0,$ and $\pi(\cdot)$ is the associated L\'evy measure. We
will assume the following:
\begin{equation} \label{condition}a_1\ge 0; \qquad
\int_0^\infty \mbox{ min}(r, r^2) \,\pi(dr) < \infty.
\end{equation}

L\'evy process exist under the weaker condition $\int_0^\infty \min(1,r^2)\,\pi(dr)<\infty$, and indeed there is a well known construction of continuous state branching processes (CSBP) as time-changes of L\'evy processes. The stronger moment condition assumed above can be thought of as a condition for this CSBP (or equivalently, this time change) not to blow up in finite time. See \cite{g}. It is easily
seen that $ \psi^{(n)} (\lambda) = \f{d^n}{d \lambda^n }
\psi(\lambda)$ exists for all $n$ and $\lambda > 0$.

Let $D$ be a domain in $\R^d$. Take $\xi_t$ to denote Brownian motion on $D$, and let $X_t$ be the $\psi$-super Brownian motion on $\R^d$. Let $X^D$ be the associated exit measure on $\partial D$. In \cite{dy042}, Dynkin constructs this object, assuming $a_1=0$. But by Dawson's Girsanov theorem (see section 10.1.2 of \cite{daw}), applying Dynkin's construction to Brownian motion killed at rate $a_1$ yields precisely the desired $\psi$-super Brownian exit measure. Note that we are using the condition $a_1\ge 0$ at this point. For $a_1<0$ and $\pi=0$ one could in fact produce a superprocess that survives forever with positive probability, by taking $D$ large enough (see \cite{enpi}). Thus the exit measure could be infinite in that case. 

 $\N_x$ will denote the excursion measure. LeGall used this measure extensively in the case $\psi(\lambda)=2\lambda^2$ (see \cite{legall2}). For the general case see Dynkin \cite{dy042}. A key benefit to working under  $\N_x$ is that genealogies simplify -- all mass descends from a single massless initial individual. The price one pays for this simplification is that  $\N_x$ is an infinite measure. Only events that involve extinction at short times receive infinite mass, so $\N_x(X^D\neq 0)<\infty$. It can be realized by having the superprocess start with initial value $\gamma\delta_x$ under a probability measure $\mathbb{P}_{\gamma\delta_x}$, sending $\gamma\downarrow 0$, and renormalizing $\mathbb{P}_{\gamma\delta_x}$ to obtain a non-trivial limit. Alternatively, the superprocess with initial value $\mu$ (under a probability $\mathbb{P}_\mu$) can be realized in terms of a Poisson random measure having $\int\N_x\mu(dx)$ as intensity. See section \ref{BrParticleDescrip}

Let $e^D_\phi = \exp-\langle X^D,\phi\rangle .$ When convenient we will also denote this
$e^D(\phi)$ or just $e(\phi)$. Let $D_k$ be a sequence
of smooth subdomains increasing to $D$ and denote $X^{D_k}$ by $X^k$ and
$e^{D_k}_\phi$ by $e^k_\phi.$ Set 
\begin{equation}\label{Nt}
\CN_t(f) = \exp\Big(-\int_0^t \,
\psi'\big(\N_{\xi_s}(1-f)\big) \,ds\Big).
\end{equation}
Since $a_1\ge 0$ we have $\psi'(\lambda)\ge 0$ when $\lambda\ge 0$, so in particular $\CN_t(f)\le 1$ for $f\ge 0$. 

We will need the following results.

\begin{lemma} \label{background} Assume \eqref{condition} and let $D$ be a domain in $\R^d.$ 
\begin{enumerate}
\item Let $\Gamma\subset\partial D$. Then $g(y)=\N_y(X^D(\Gamma)>0)$ satisfies  $\frac12\Delta g = \psi(g)$.
\item Let $g$ be a non-negative solution to $\frac12\Delta g = \psi(g)$, and let $D_k$ be an
increasing sequence of smooth subdomains of $D$. Then for each $k$,
\[ \N_x (1-e^k_g) = g(x). \]

\item $\N_x (\langle X^D, \phi\rangle  e^D_f) = E_x( \phi ( \xi_{\tau_D}) 
\CN_{\tau_D}(e^D_f))$ for $f,\phi\ge 0$.
\end{enumerate}
\end{lemma}
\begin{proof}
Follows from  Theorem 4.2.1 in \cite{dy02}, Theorem 1.1. in Chapter 4 of \cite{dy042} and Theorem 11.7.1 in \cite{daw}.\\ 
\end{proof}

For $m\ge 2$ an integer,  $y\in \R^d,$ and measurable $\phi$ we define
\begin{equation} \label{b}
b(m,\phi, y) = (-1)^{m} \psi^{(m)} \big(\N_{y}(1-e_\phi)\big).
\end{equation}
Note that for $\lambda > 0$
\begin{equation}\label{psi^(n)}
\psi^{(m)} (\lambda)=
\begin{cases}
a_1 + 2a_2 \lambda +\int_0^\infty r[1- e^{-\lambda r}]\, \pi(dr), & m=1\\
2a_2 + \int_0^\infty r^2 e^{-\lambda r} \,\pi(dr), & m=2\\
(-1)^m\int_0^\infty\, r^m e^{-\lambda r}  \,\pi(dr), & m \ge 3,
\end{cases}
\end{equation}
so in particular, $b(m,\phi,y)$ is well defined and non-negative, for $m\ge 2$ and for any $\phi\ge 0$ such that $\N_y(\langle X^D,\phi\rangle>0)>0$. It is decreasing in $\phi$. 

We will need the following assumption on $\psi$ for certain results, which, among other things, makes the latter qualification unnecessary. 
$$ \mbox{ (A1)  \qquad $\exists \lambda_0>0$ such that $\int_1^\infty  e^{r\lambda_0}\,\pi(dr) <\infty$.} $$
Under (A1) we have that $\int_0^\infty  r^n\, \pi(dr)<\infty$ for $n\ge 2$, and by dominated convergence, 
\begin{equation}\label{psi_series}
\psi(\lambda)=a_1\lambda+a_2\lambda^2 + \sum_{j=2}^\infty (-1)^j\frac{\lambda^j}{j!}\int_0^\infty r^j\,\pi(dr)
\end{equation}
for $\lambda\le\lambda_0$. Thus in this case, \eqref{psi^(n)} will hold as well for $\lambda=0$.

Recursive Palm formulae for moments have a long history, and in this context are due to Dynkin -- see Theorem 1.1 of \cite{dy042} or Theorem 4.1 of \cite{dy041}. We work instead with a formulation along the lines of Lemma 2.6 in \cite{sv1}.
 
\begin{lemma} \label{palmformula}
Assume \eqref{condition} and let $N= \{1,2,3, \ldots n\}, n \geq 2.$ 
Let $D \subset \R^d$ be a
domain, $\phi\ge 0$, and let $\xi$ be a Brownian motion in $D$ with exit time
$\tau$. Let $\{ v_i \}$ be a family of positive measurable functions. Then 
$$
\N_x( e_\phi \prod_{i \in N}  \langle X^D , v_i \rangle ) =  
E_x
\Big( \sum_{\substack{\beta \in \CP(N)\\|\beta|\ge 2}}
\int_0^\tau  \,{\mathcal N}_t(e_\phi) 
b(|\beta|, \phi, \xi_{t}) \prod_{A \in \beta }\N_{\xi_{t}}
(e_\phi \prod_{i \in A} \langle X^D,v_i\rangle  ) \,dt \Big).
$$
\end{lemma}

\begin{proof} Let  $N^{\star} = \{ 2,3, \ldots, N \}$.
\begin{align*}
\N_x( e_\phi \prod_{i \in N} \langle X^D , v_i \rangle ) 
&= \N_x ( \langle X^D,v_1\rangle 
e_\phi \prod_{i \in N^{\star}} \langle X^D , v_i \rangle )\\ 
&= (-1)^{n-1}\left.
\f{\partial}{\partial \lambda_2}\f{\partial}{\partial \lambda_3}\cdots
\f{\partial}{\partial \lambda_n}
\right|_{\lambda_2=\ldots=\lambda_n=0}\hspace{-0.5in}
\N_x (\langle X^D,v_1\rangle  e( \phi + \sum_{i=2}^n \lambda_i v_i )),
\end{align*}
where we differentiate under the integral sign using monotone convergence. 
Using (c) of Lemma \ref{background} i.e. the one-dimensional case of the
Palm formula, the above expression equals
\begin{align*} 
&(-1)^{n-1}\left.
\f{\partial}{\partial \lambda_2}\f{\partial}{\partial \lambda_3}\cdots
\f{\partial}{\partial \lambda_n} 
\right|_{\lambda_2=\dots=\lambda_n=0}
\hspace{-0.5in} E_x \big( v_1(\xi_\tau) {\mathcal N}_\tau(e(\phi +
\sum_{i=2}^n \lambda_i v_i ) ) \big)
\end{align*}
Differentiating this expression under the integral sign (given the conditions on $\pi$ this can be easily justified) and using the definition of $b$, this equals
\begin{align}
\label{first}
&E_x \left ( v_1(\xi_\tau) {\mathcal N}_\tau(e_\phi ) \sum_{\sigma
\in \CP(N^{\star})}\prod_{j =1}^{| \sigma |} \int_0^\tau 
\sum_{\beta_j \in \CP(\sigma(j)) } (-1)^{| \beta_j | + 1} \psi^{| \beta_j
| +1} \N_{\xi_{t_j}}(1-e_\phi) \times \right. \nonumber \\
&\hspace{2.5in} \left . \times\prod_{A \in \beta_j }\N_{\xi_{t_j}} (e_\phi
\prod_{i \in A} \langle X^D,v_i\rangle  )\,dt_j \right)\nonumber \\
&\quad=E_x \left ( v_1(\xi_\tau)
{\mathcal N}_\tau(e_\phi ) \sum_{\sigma \in \CP(N^{\star})}
\prod_{j=1}^{| \sigma |} \int_0^\tau \sum_{\beta_j \in \CP(\sigma(j)) } 
b(|\beta_j | + 1 ,\phi,\xi_{t_j}) \times \right .\\ 
& \left . \hspace{2.5in} \times \prod_{A \in \beta_j }\N_{\xi_{t_j}} (e_\phi \prod_{i
\in A} \langle X^D,v_i\rangle  )\,dt_j \right). \nonumber 
\end{align}

Standard integration manipulations then show that the above equals
\begin{align*}
&E_x \left ( v_1(\xi_\tau) {\mathcal
N}_\tau(e_\phi ) \sum_{\sigma \in \CP(N^{\star})} \sum_{k=1}^{|
\sigma |} \int_0^\tau \sum_{\beta_k \in \CP(\sigma(k)) }
b(|\beta_k|+1, \phi, \xi_{t_k}) \prod_{A \in \beta_k }\N_{\xi_{t_k}}
(e_\phi \prod_{i \in A} \langle X^D,v_i\rangle  ) \times \right .\\ 
& \left. \hspace{1.5in}\times \prod_{j\neq k}\Big(\int_{t_k}^\tau \sum_{\beta_j
\in \CP(\sigma(j)) } b(|\beta_j| +1, \phi, \xi_{t_j}) \prod_{A \in \beta_j
}\N_{\xi_{t_j}} (e_\phi \prod_{i \in A} \langle X^D,v_i\rangle  ) \,dt_j\Big)\,dt_k\right)\\
&\quad=E_x
\left ( \sum_{\sigma \in \CP(N^{\star})} \sum_{k=1}^{| \sigma |}
\int_0^\tau \sum_{\beta_k \in \CP(\sigma(k)) } b(|\beta_k| +1, \phi,
\xi_{t_k}) \right.  \prod_{A \in \beta_k }\N_{\xi_{t_k}} (e_\phi \prod_{i \in A} \langle X^D,v_i\rangle  )  \times  \\ & \left.  \hspace{1.5in}\times E_{x} \left (
v_1(\xi_\tau) {\mathcal N}_\tau(e_\phi ) \prod_{j\neq
k}\int_{t_k}^\tau \sum_{\beta_j \in \CP(\sigma(j)) } b(|\beta_j| +1, \phi,
\xi_{t_j}) \times\right.\right.\\
&\left.\left.\hspace{2.5in}\times\prod_{A \in \beta_j }\N_{\xi_{t_j}} (e_\phi \prod_{i \in A}
\langle X^D,v_i\rangle  ) \mid {\CF}_k \, dt_j\right )\,dt_k\right)
\end{align*}
where  $\CF_k$ denote the filtration of $\xi_{t_k}.$ Applying the
Markov Property at time $t_k$, this equals 
\begin{align*} 
& E_x
\left ( \sum_{\sigma \in \CP(N^{\star})} \sum_{k=1}^{| \sigma |}
\int_0^\tau  \,{\mathcal N}_{t_k}(e_\phi ) \sum_{\beta_k \in \CP(\sigma(k)) } b(|\beta_k| +1, \phi,
\xi_{t_k}) \right.  \prod_{A \in \beta_k }\N_{\xi_{t_k}} (e_\phi \prod_{i \in A} \langle X^D,v_i\rangle  )  \times  \\ & \left.  \hspace{1.0in}\times E_{\xi_{t_k}} \left (
v_1(\xi_\tau) {\mathcal N}_\tau(e_\phi ) \prod_{j\neq
k}\int_{0}^\tau \sum_{\beta_j \in \CP(\sigma(j)) } b(|\beta_j| +1, \phi,
\xi_{t_j}) \times\right.\right.\\
&\left.\left.\hspace{2.0in}\times\prod_{A \in \beta_j }\N_{\xi_{t_j}} (e_\phi \prod_{i \in A}
\langle X^D,v_i\rangle  )\,dt_j\right )\,dt_k\right).
\end{align*}
Setting $M=\sigma(k)$ and summing over $N^{\star}$ this becomes
\begin{align*}
&E_x \left ( \sum_{\emptyset \neq M \subset N^{\star}}
\int_0^\tau \, {\mathcal N}_t(e_\phi ) \sum_{\gamma \in \CP(M) }
b(|\gamma|+1, \phi, \xi_{t}) \prod_{A \in \gamma }\N_{\xi_{t}}
(e_\phi \prod_{i \in A} \langle X^D,v_i\rangle  ) \times \right .\\  
&\hspace{1.0in}\times \sum_{\beta \in \CP(N^{\star} \setminus M )}
E_{\xi_{t}} \left ( v_1(\xi_\tau) {\mathcal N}_\tau(e_\phi )
\prod_{j=1}^{| \beta | }\int_{0}^\tau \sum_{\beta_j \in
\CP(\sigma(j)) } b(|\beta_j| +1, \phi, \xi_{s}) \times\right.\\
&\left.\left.\hspace{2.0in}\times \prod_{A \in
\beta_j }\N_{\xi_{s}} (e_\phi \prod_{i \in A} \langle X^D,v_i\rangle  )\,ds\right )
\,dt\right).
\end{align*}
Using the identity obtained in
\eqref{first}, this equals
\begin{align*}
&E_x \left ( \sum_{\emptyset \neq M \subset N^{\star}}
\int_0^\tau \,{\mathcal N}_t(e_\phi ) \sum_{\beta \in \CP(M) }
b(|\beta| +1, \phi, \xi_{t}) \prod_{A \in \beta }\N_{\xi_{t}}
(e_\phi \prod_{i \in A} \langle X^D,v_i\rangle  ) \times\right . \\ 
& \left. \hspace{1.5in} \times \N_{\xi_t}( e_\phi \prod_{i \in N \setminus
M} \langle  X^D , v_i \rangle ) \,dt\right ).
\end{align*}

For any $\beta \in \CP(N)$ with $|\beta|\ge 2$ we can realize a term of the above 
expression, letting $N\setminus M$ be the element of $\beta$ containing 1, and 
letting $\gamma$ be the restriction of $\beta$ to $M$. Therefore
$$
\N_x( e_\phi \prod_{i
\in N} \langle  X^D , v_i \rangle ) 
= E_x \Big( \sum_{\substack{\beta \in \CP(N)\\|\beta|\ge 2}}
\int_0^\tau\, {\mathcal N}_t(e_\phi ) b(|\beta|, \phi,
\xi_{t}) \prod_{A \in \beta }\N_{\xi_{t}} (e_\phi \prod_{i \in A}
\langle X^D,v_i\rangle  )\,dt \Big).
$$
\end{proof}

\begin{lemma} \label{expmom} Assume (A1). Let $D$ be a domain in $\R^d$ satisfying $\sup_{x \in D} E_x (\tau_D) < \infty,$
where $\tau_D$ is the exit time from $D$ for Brownian motion. Then there exists $\lambda >0$
such that 
\[ \sup_{x \in D} \N_x (\exp(\lambda \langle X^D, 1 \rangle) -1) < \infty. \]
\end{lemma}

\begin{proof}
Let $c_n = \sup_{x \in D} \N_x( \langle X^D,1 \rangle ^n).$  As $a_1 \geq 0$, we have $c_1 \leq  1$ by (c) of Lemma \ref{background}. By Lemma \ref{palmformula}, with $\phi = 0$ we have the following
recursion relation, for $n \geq 2$,
\begin{eqnarray} \label{srec1}
c_n &\leq&  E_x \Big( \sum_{\substack{\beta \in \CP(N)\\|\beta|\ge 2}}
\int_0^\tau \,
b(|\beta|,0, \xi_{t}) \prod_{A \in \beta } c_{|A|}\,dt \Big)\\
&\leq & K \sum_{\substack{\beta \in \CP(N)\\|\beta|\ge 2}}
m_{|\beta|} \prod_{A \in \beta } c_{|A|}\\
&=& K \sum_{j=2}^n \frac{m_j}{j!}  \sum_{\substack{i_1,i_2,\ldots, i_j \geq 1 \\ i_1 + i_2 + \ldots + i_j = n }} \frac{n!}{i_1! i_2! \ldots i_j!} \prod_{i=1}^j c_{i_k}
\end{eqnarray}
where $K >0$, $m_2 = 2a_2 + \int_0^\infty  r^2\,\pi(dr)$ and $m_k = \int_0^\infty  r^k \,\pi(dr)$ for $k \geq 3$. 
 For $N \geq 1$, let $g_N:[0,\infty) \rightarrow [0,\infty)$ given by 
\begin{equation} g_N(\lambda) =  \sum_{n=1}^N \frac{c_n \lambda^n}{n!} , \,\, \lambda >0. \end{equation}
Using \eqref{srec1}, by an inductive argument it is easy to see that $c_n < \infty$ for all $n \geq 2$, so $g_N$ is well defined for all $N \geq 1$. Set $\tilde{\psi} (u ) = \sum_{j=2}^\infty \frac{m_j}{j!} u^j=\psi(u)-a_1u-a_2u^2$ for $0\le u<\lambda_0$. Then using \eqref{srec1} again,  for $\lambda >0$
\begin{eqnarray}\label{inequ}
g_N(\lambda) &\leq& \lambda + K \sum_{n=2}^N  \lambda^n \sum_{j=2}^n \frac{m_j}{j!}  \sum_{\substack{i_1,i_2,\ldots, i_j \geq 1 \\ i_1 + i_2 + \ldots + i_j = n }}  \prod_{i=1}^j \frac{c_{i_k}}{i_k!} \nonumber \\
&\leq& \lambda + K  \sum_{j=2}^N \frac{m_j}{j!}  \sum_{1 \leq i_1,i_2,\ldots, i_j \leq N}  \prod_{i=1}^j \frac{c_{i_k} \lambda^{i_k}}{i_k!} \nonumber \\
&=& \lambda + K  \sum_{j=2}^N \frac{m_j}{j!} \left ( g_N(\lambda) \right )^j \nonumber \\
&\leq& \lambda + K  \tilde{\psi}(g_N(\lambda)), 
\end{eqnarray}
provided $g_N(\lambda)<\lambda_0$. By the assumptions on $\psi$, $\tilde{\psi}$ is infinitely differentiable on $[0,\lambda_0)$ with $\psi''\ge 0$ and $\tilde{\psi} (0) = 0= \tilde{\psi^\prime}(0)$.  
So we may find $x_0<\lambda_0$ sufficiently small that the line through $(x_0,\tilde\psi(x_0))$ with slope $1/K$ is secant to the graph of $g_N$. That is, 
\begin{equation}\label{tineq}
\text{ $\tilde{\psi}(x) > \tilde \psi(x_0)+\frac{x-x_0}{K}$ for $x\in[0,x_0)$.}
\end{equation}
Let $\lambda_1=x_0-K\tilde\psi(x_0)<x_0<\lambda_0$. Since 
$$
0=\tilde\psi(0)>\tilde\psi(x_0)-\frac{x_0}{K}=-\frac{\lambda_1}{K},
$$
we also have $\lambda_1>0$. 

We claim that 
\begin{equation}\label{gNbound}
g_N(\lambda)\le x_0\quad\forall\lambda\in[0,\lambda_1].
\end{equation} 
To see this, observe that $g_N$ is continuous and strictly increasing, with $g_N(0)=0$. So if \eqref{gNbound} fails to hold, there will be a unique $\lambda<\lambda_1$ with $g_N(\lambda)=x_0$. But by \eqref{inequ}, 
$$
x_0=g_N(\lambda)\le \lambda+K\tilde\psi(g_N(\lambda))<\lambda_1+K\tilde\psi(x_0)=x_0,
$$
which is impossible. 

Now just let $N\to\infty$ in \eqref{gNbound} to complete the proof.
\end{proof}
The authors are grateful to Amram Meir, who showed them how to construct this type of argument. For example, see \cite{mm}.

\subsection{Potentials} \label{potential}
We will need certain results concering potentials of specific partial differential equations. These potentials will be used to describe the exit measure conditioned to  hit certain points on the boundary of $D$.

Let $N = \{1,2,\ldots,n \}$ be as before. Then for every non-empty
subset $A\subset N$ let us suppose we are given a solution $u^A>0$ to
the equation $\f{1}{2}{\Delta u} = \psi(u)$ in $D$. For convenience we also set $u^A=0$ for
$A=\emptyset$. Define
\begin{equation} \label{uvrel}
v_A =
\sum_{\substack{N \setminus A \subset B \subset N}}
(-1)^{| A | + | B | + n + 1 } u^B \quad \mbox{ and }
\quad v^A =
\sum_{\substack{ \emptyset \not = B \subset A}}
(-1)^{ | B |  + 1 } u^B. 
\end{equation}
We shall assume that 
\begin{equation}\label{vPositivity}
\text{$v_A \geq 0$ for all $\emptyset \neq A \subset N$.}
\end{equation} 
The example to keep in mind is as follows: for $\Gamma_1,\dots,\Gamma_n 
\subset \partial D$, let  $$u^A(x)=\N_x(\text{$X^D$ charges $\bigcup_{i\in A} \Gamma_i$}),$$ $$ v^A(x)=\N_x(\text{$X^D$ charges $\Gamma_i$ for every $i\in A$  }), \mbox{ and }$$
$$v_A(x)=\N_x(\text{$X^D$ charges $\Gamma_i$ but not $\Gamma_j$, for every 
$i\in A$ and $j\notin A$}).$$ 
\eqref{uvrel} holds in this case, by a simple inclusion-exclusion argument (see also Lemma 5.1 in \cite{sv1} and Section 4 in \cite{sv2}).

\begin{lemma} \label{harmonic}
Assume \eqref{condition} and \eqref{vPositivity}. Then
\begin{enumerate}
\item $\displaystyle u^A = 
\sum_{\substack{B \subset N \\ A \cap B \neq \emptyset }} v_B=\sum_{\emptyset\neq B\subset A}(-1)^{|B|+1}v^B$ and 
$\displaystyle v^A = 
\sum_{A\subset B\subset N} v_B$.
\item For $A\neq\emptyset$, 
\begin{equation} \label{v_A} \displaystyle \f{1}{2} \Delta v_A - \psi'(u^N)v_A = -
\sum_{j=2}^{\infty} \f{b(j,u^N, \cdot)}{j!} 
\sum_{\substack{C_1 \cup C_2 \ldots \cup C_j = A \\ C_i \neq \emptyset}}\, 
\prod_{i=1}^j v_{C_i}.
\end{equation}
\item Assume also (A1). The following then holds at any point where $u^{A}<\lambda_0$:
\begin{equation} \label{v^A}
 \displaystyle \f{1}{2} \Delta v^A -\phi(u^A, v^A) v^A = -\sum_{j=2}^{\infty} \f{(-1)^j \psi^{(j)}(0) }{j!} 
\sum_{\substack{ C_1 \cup C_2 \ldots \cup C_j = A\\ \emptyset \neq C_i \neq A}}\, 
 (-1)^{|A|} \prod_{i=1}^j v^{C_i} (-1)^{|C_i|}
\end{equation}
where $ \phi(u^A,v^A) = \f{\psi(u^A + (-1)^{|A|}v^A) - \psi(u^A)}{(-1)^{|A|} v^A}\ge 0. $
\newline In particular, if $\int_1^\infty e^{r\lambda}\, \pi(dr)<\infty$ for every $\lambda>0$ then \eqref{v^A} holds without restriction. 
\end{enumerate}
\end{lemma}
Note that all terms in the sum from \eqref{v_A} have the same sign, whereas those from \eqref{v^A} vary in sign. The extra conditions in part (c) arise because of the possibility of conditional convergence. Part (b) describes the functions we're primarily interested in, but part (c) will be useful for asymptotics. 

\begin{proof}
The proof of first equality in part (a) is the same as that of (a) of Lemma 4.1 in
\cite{sv2}.  As indicated in Remark 4.2 of \cite{sv2},  the proof of the second and third equality follow similarly.

We will show part (b) first. 

{\bf Proof of (\ref{v_A}) :}
\begin{align*}
\f{1}{2} \Delta v_A 
&=\sum_{\substack{N \setminus A \subset B \subset N
}} (-1)^{| A | + | B | + n + 1 } \f{1}{2} \Delta u^B
\\ &=\sum_{\substack{N \setminus A \subset B \subset N
}} (-1)^{| A | + | B | + n + 1 } \psi(u^B)\\ 
&=\sum_{\substack{N \setminus A \subset B \subset N
}} (-1)^{| A | + | B | + n + 1 } \left (a_1 u^B + a_2
(u^B)^2 + \int_0^\infty  (e^{-r u^B} + r u^B -1) \,\pi (dr)\right ) \\
&= a_1 v_A + 2a_2u^Nv_A - a_2\sum_{ \substack{C \cup C^{\prime} =A \\ C,
C^{\prime} \neq \emptyset }} v_C v_{C^{\prime}} +\\
& \hspace{1.5in}+ \sum_{\substack{N
\setminus A \subset B \subset N}} (-1)^{| A
| + | B | + n + 1 } \int_0^\infty (e^{-r u^B} + r
u^B-1 ) \, \pi (dr)
\end{align*}
In the last line we have used the definition of $v_A$ 
and (b) of Lemma 4.1 of \cite{sv2} (which will be recognized  as the current lemma in the case
$\psi(u) = 2u^2$.) Using the expression for $\psi'(\cdot)$ from
\eqref{psi^(n)}, and that of $v_A$ we get that
\begin{align}
\f{1}{2} \Delta v_A 
&- \psi'(u^N)v_A = - \int_0^\infty
 (-r v_A e^{- u^N r} + rv_A ) \, \pi(dr) - a_2\sum_{
 \substack{C \cup C^{\prime} =A \\ C, C^{\prime} \neq \emptyset }} v_C
 v_{C^{\prime}} +
 \nonumber \\ 
& \hspace{1.5in}+ \sum_{\substack{N \setminus A \subset B \subset N}}
 (-1)^{| A | + | B | + n + 1 }
 \int_0^\infty (e^{-r u^B} + ru^B -1)\, \pi(dr)\nonumber \\ 
&= -a_2\sum_{ \substack{C \cup C^{\prime} =A \\ C, C^{\prime} \neq
 \emptyset }} v_C v_{C^{\prime}} 
+\sum_{\substack{N \setminus A
 \subset B \subset N}} (-1)^{| A | + | B
 | + n + 1 }\int_0^\infty  (e^{-ru^B} + r u^Be^{- r u^N }-1)\,\pi(dr).
 \label{step}
 \end{align}
 Consider the last term of this expression. $A\neq\emptyset$ so by Lemma \ref{settheory},
 $$
 \sum_{\substack{N \setminus A
 \subset B \subset N}} (-1)^{| A | + | B
 | + n + 1 }\int_0^\infty  (1-ru^Ne^{-ru^N} -e^{- r u^N })\,\pi(dr)=0.
 $$
 Therefore
  \begin{align*}
 \sum_{\substack{N \setminus A
 \subset B \subset N}} &(-1)^{| A | + | B
 | + n + 1 }\int_0^\infty  (e^{-ru^B} + r u^Be^{- r u^N }-1)\,\pi(dr)\\
 &=\sum_{\substack{N \setminus A
 \subset B \subset N}} (-1)^{| A | + | B
 | + n + 1 }\int_0^\infty  (e^{-ru^B} + r u^Be^{- r u^N }-ru^Ne^{-ru^N}-e^{-ru^N})\,\pi(dr)\\
&= \sum_{\substack{N
 \setminus A \subset B \subset N}} (-1)^{| A
 | + | B | + n + 1 }\int_0^\infty e^{-ru^N}( e^{r
 (u^N - u^B)} -r(u^N- u^B)-1)\,\pi(dr)\\ 
&= \sum_{\substack{N \setminus A \subset B \subset N\\
}} (-1)^{| A | + | B | + n + 1
 }\int_0^\infty  e^{-ru^N} \sum_{j=2}^{\infty} \f{r^j}{j!}
 (u^N-u^B)^j\pi(dr).
 \end{align*}
By part (a), $\displaystyle u^N-u^B = \sum_{\substack{ \emptyset\neq C \subset N\\ C \cap B = \emptyset}} v_C$. So by monotone convergence we have that the above is 
\begin{eqnarray}
\label{third term}
&=& \sum_{\substack{N \setminus A \subset B \subset N}} (-1)^{| A | + | B | + n + 1 } \sum_{j=2}^{\infty}\int_0^\infty  e^{-ru^N} \f{r^j}{j!}
 \Big(\sum_{\substack{ \emptyset\neq C \subset N\\ C \cap B = \emptyset}} v_C\Big)^j\, \pi(dr)\nonumber\\
&=&
\sum_{j=2}^{\infty}\int_0^\infty  e^{-ru^N} \f{r^j}{j!}
\sum_{\substack{N \setminus A \subset B \subset N\\ B \neq
 \emptyset}} (-1)^{| A | + | B | + n + 1 }
 \Big(\sum_{\substack{ \emptyset\neq C \subset N\\ C \cap B = \emptyset}} v_C\Big)^j\,\pi(dr).
\end{eqnarray}
Using standard multinomial expansions and Lemma \ref{settheory}  we observe that $j$-th summand is, 
\begin{eqnarray*}
&= &\sum_{\substack{N \setminus A \subset B \subset N}} (-1)^{| A | + | B | + n + 1 }
 \Big(\sum_{\substack{ \emptyset\neq C \subset N\\ C \cap B = \emptyset}} v_C\Big)^j\\
&=&\sum_{\substack{N \setminus A \subset B \subset N
}} (-1)^{| A | + | B | + n + 1 }
 \sum_{\substack{C_1, C_2, \dots C_j\\ \emptyset\neq C_i \subset
 N \setminus B} } \prod_{i=1}^j v_{C_i} \\ 
&= &\sum_{\substack{C_1, C_2, \dots C_j\\ \emptyset\neq C_i \subset
 N} }(\prod_{i=1}^j v_{C_i} )\sum_{N \setminus A  \subset B \subset N 
 \setminus \cup_i^j C_i} (-1)^{| A | + | B | + n + 1 } \\ 
&= &\sum_{\substack{C_1, C_2, \dots C_j\\ \emptyset\neq C_i \subset
 N} }(\prod_{i=1}^j v_{C_i}) (-1)^{| A | + n + 1 } (-1)^{n-| A
 |} 1_{ A = \cup_i^j C_i}\\
&=& -\sum_{\substack{C_1 \cup C_2 \ldots \cup
 C_j = A\\C_i \neq \emptyset}}\prod_{i=1}^j v_{C_i}.
 \end{eqnarray*}
Together with \eqref{step} and \eqref{third term} this implies \eqref{v_A}. It is clear from the proof that the sum in \eqref{v_A} converges. 

{\bf Proof of (\ref{v^A}) : }
\begin{eqnarray*} 
\f{1}{2} \Delta v^A 
&=&\sum_{\substack{ \emptyset \not =  B \subset A
}} (-1)^{ | B |  + 1 } \f{1}{2} \Delta u^B \\ 
&=&\sum_{\substack{\emptyset \not = B \subset A
}} (-1)^{| B | + 1 } \psi(u^B).
\end{eqnarray*}
Observe that $u^B\le u^A$ by part (a) and \eqref{vPositivity}. So by \eqref{psi_series}, we can expand $\psi$ in a series, at any point where $u^A<\lambda_0$. Therefore
\begin{eqnarray}\label{inter0}
\f{1}{2} \Delta v^A 
&=&\sum_{\substack{\emptyset \not = B \subset A
}} (-1)^{| B | + 1 } \left (a_1 u^B +   \sum_{j=2}^\infty \f{\psi^{(j)}(0)}{j!}(u^B)^j \right )\nonumber \\
&=&  a_1 v^A + \sum_{j=2}^\infty \f{(-1)^j\psi^{(j)}(0)}{j!}
\sum_{\substack{\emptyset \not \neq B \subset A}} (-1)^{  | B | +  1 }(-u^B)^j.
\end{eqnarray}
For $j \geq 2$,  we have by part (a) that
\begin{eqnarray}
\lefteqn{\sum_{\substack{\emptyset \not \neq B \subset A}} (-1)^{  | B | +  1 }(-u^B)^j} \nonumber \\
&=& 
\sum_{\substack{\emptyset \not \neq B \subset A}} (-1)^{  | B | +  1 }\Big(\sum_{\emptyset \not = C \subset B} v^C (-1)^{|C|} \Big)^j \nonumber \\
&=&   
\sum_{\substack{\emptyset \not \neq B \subset A}} (-1)^{  | B | +  1 } \sum_{ \substack{\emptyset \neq C_1 ,C_2, \ldots  C_j \subset B }} \prod_{i=1}^j v^{C_i} (-1)^{|C_i|}  \nonumber \\
&=& (-1)^{  | A | +  1 } \sum_{ { \emptyset \neq C_1,
C_2, \ldots, C_j \subset A}} \prod_{i=1}^j v^{C_i} (-1)^{|C_i|} 1(A= \cup_{i=1}^j C_i) \nonumber \\
&& (\mbox{using Lemma \ref{settheory} }) \nonumber\\
\label{inter2}
&=& (-1)^{  | A | +  1 } \sum_{ \substack{ A= \cup_{i=1}^j C_i \\ \emptyset \neq C_1,
C_2, \ldots, C_j \neq A}} \prod_{i=1}^j v^{C_i} (-1)^{|C_i|} \nonumber \\ && + (-1)^{  | A | +  1 } \sum_{k=1}^j {j \choose k} {((-1)^{|A|}v^A)}^k \left (\sum_{ { \emptyset \neq C_1,
C_2, \ldots, C_{j-k}  \neq  A}} \prod_{i=1}^{j-k} v^{C_i} (-1)^{|C_i|}\right )  \nonumber\\
&=& (-1)^{  | A | +  1 } \sum_{ \substack{ A= \cup_{i=1}^j C_i \\ \emptyset \neq C_1,
C_2, \ldots, C_j \neq A}} \prod_{i=1}^j v^{C_i} (-1)^{|C_i|}  \nonumber \\ && +  (-1)^{  | A | +  1 } \sum_{k=1}^j {j \choose k} {((-1)^{|A|}v^A)}^k  ( -u^A +-(-1)^{|A|}v^A )^{j-k}\nonumber\\
&=&  (-1)^{  | A | +  1 }\sum_{ \substack{ A= \cup_{i=1}^j C_i \\ \emptyset \neq C_1,
C_2, \ldots, C_j \neq A}} \prod_{i=1}^j v^{C_i} (-1)^{|C_i|} +   (-1)^{  | A | +  1 } (-u^A)^j  - ( -u^A +-(-1)^{|A|}v^A )^{j}.\nonumber\\
&=&  \sum_{\emptyset \neq B \subset A} (-1)^{|B|+1}  \left ( \sum_{\emptyset \neq C \subset B, C \neq B} v^C (-1)^{|C|}  \right )^j +   (-1)^{  | A | +  1 } (-u^A)^j  - ( -u^A +-(-1)^{|A|}v^A )^{j}.\nonumber\\
\end{eqnarray}
Using \eqref{inter0} and \eqref{inter2} we have
\begin{eqnarray*} 
\f{1}{2} \Delta v^A &=&  a_1 v^A +(-1)^{|A| +1 }\sum_{j=2}^\infty \f{(-1)^j\psi^{(j)}(0)}{j!}
  \sum_{ \substack{ A= \cup_{i=1}^j C_i \\ \emptyset \neq C_1,
C_2, \ldots, C_j \neq A}} \prod_{i=1}^j v^{C_i} (-1)^{|C_i|} +     \nonumber \\ &&  +  (-1)^{|A|+1}\sum_{j=2}^\infty \f{(-1)^j\psi^{(j)}(0)}{j!} (-u^A)^j  - ( -u^A +-(-1)^{|A|}v^A )^{j}\\ 
\end{eqnarray*}
\begin{eqnarray*}
&=& a_1 v^A + (-1)^{|A| +1 }\sum_{j=2}^\infty \f{(-1)^j\psi^{(j)}(0)}{j!}
  \sum_{ \substack{ A= \cup_{i=1}^j C_i \\ \emptyset \neq C_1,
C_2, \ldots, C_j \neq A}} \prod_{i=1}^j v^{C_i} (-1)^{|C_i|} +  
\nonumber \\ &&  +  (-1)^{|A|+1}\int_0^\infty (e^{-ru^A} - e^{-r( u^A +(-1)^{|A|}v^A )})\, \pi(dr)\\ 
&=& - (-1)^{|A|}\sum_{j=2}^\infty \f{(-1)^j\psi^{(j)}(0)}{j!}
 (-1)^{  | A | +  1 } \sum_{ \substack{ A= \cup_{i=1}^j C_i \\ \emptyset \neq C_1,
C_2, \ldots, C_j \neq A}} \prod_{i=1}^j v^{C_i} (-1)^{|C_i|} +  
\nonumber \\ &&  +  (-1)^{|A|}( \psi(u^A + (-1)^{|A|}v^A)  - \psi(u^A)) \\ 
\end{eqnarray*}
which gives \eqref{v^A}. Positivity of $\phi$ follows from monotonicity of $\psi'$. 
\end{proof}

\section{The Branching particle description} \label{BrParticleDescrip}
In this section we shall define a martingale change of measure which will represent various 
conditionings of the exit measure of $\psi$ super-Brownian motion. For each such conditioning we shall 
also present a branching backbone representation, in which the conditioned process is realized as an 
unconditioned superprocess with immigration of mass along a branching tree.

Let $\emptyset\neq A \subset N = \{1,2 \ldots, n \}$ and let $u^A\ge 0$ and $v_A\ge 0$ be as in
Section \ref{potential}. Specific examples will be described later, but even at this level of generality we can use these functions to define an associated martingale. Define 
\begin{equation}
\label{martingale}
\breve{M}_k = \exp (-\langle X^k , u^N\rangle ) \sum_{m=1}^\infty \f{1}{m!}
\, \sum_{\substack{ \scriptstyle C_1\cup\ldots \cup C_m =N \\ C_i \neq
\emptyset }} \,\prod_{i=1}^m \langle X^k, v_{C_i}\rangle \ge 0.
\end{equation}

In Lemma 4.3 of \cite{sv2}, it is shown that this sum is finite, and 
$$\breve{M}_k= 
\sum_{A \subset N} (-1)^{| A|} \exp(-\langle X^k, u^A\rangle )
$$
(note that the $A=\emptyset$ term equals 1, by our convention that $u^\emptyset=0$).

From this or otherwise it can be easily checked that 
$\breve{M}_k$ is a martingale under $\N_x$, with respect to the filtration
$\CF_k$.  So for $\Phi$ measurable with respect to
$\CF_k$ let $\breve{\M}_x = \f{1}{v_N(x)} {\N}_x (\Phi \breve{M}_k )$ be
the associated Girsanov transformation -- that is, the  $h$-transform, for the ``harmonic'' function 
$$
h(\mu)=\sum_{A \subset N} (-1)^{| A|} \exp(-\langle \mu, u^A\rangle ).
$$
Dynkin has developed a general framework for such $h$-transforms of superprocesses, which he calls {\sl $X$-harmonic functions}. See \cite{dy041}.

{\bf Branching backbone : }\label{BrBackbone}
We describe the direct construction of  $\breve{\M}_x$ in terms
of a branching particle system tracing out a backbone along which
mass gets created. Let $\breve Lf = L_{\psi'\circ u^N}f=
\f{1}{2} \Delta f - \psi'(u^N)f.$ Essentially we will constructively generate a measure
$\breve{\N}_x$ and then show that it agrees with $\breve{\M}_x$. 
$\breve{\N}_x$ will have a branching backbone $\Upsilon$ equipped
with mass creation/immigration both along the branches and at the nodes. Each branch will be labeled by a nonempty subset $A$ of $N$. The mass created along the backbone will evolve as an unconditioned superprocess but with a modified killing rate and L\'evy measure. The first order of business is to give a more precise description of the various components of this process. 

{\bf Evolution of Mass:} Once mass has been created, it evolves as an unconditioned super-Brownian motion but with a modified branching law. More formally, for $\Phi$ measurable with respect to $\CF_k$,
let $\tilde{\N}_y(\Phi) = \N_y(\Phi e^k_{u^N})$. 
We will show that this measure indeed describes a superprocess.

For any $\lambda>0$
and $y \in D$ define 
$\breve\psi (y,\lambda) = \psi( u^N(y) + \lambda) -
\psi(u^N(y))$.
It is easily checked that 
$$
\breve\psi(y,\lambda)=\breve a_1(y)\lambda+ a_2\lambda^2+\int_0^\infty e^{-u^N(y)}(e^{-\lambda r}-1+\lambda r)\, \pi(dr)
$$
where $\breve a_1(y)=a_1+2a_2u^N(y)+\int_0^\infty r(1-e^{-ru^N(y)})\, \pi(dr)$. This is therefore a spatially varying branching law, of the form \eqref{psi}. But now in places where $u^N$ is large, we will have that the killing rate $\breve a_1$ becomes large too (provided $a_2\neq 0$), and the L\'evy branching measure $\pi(dr)e^{-ru^N}$ becomes small. 

The following is a special case of Dawson's Girsanov formula (see Theorem 7.2.2. in \cite{daw}),

\begin{lemma}\label{Dawson Girsanov} Assume \eqref{condition}. 
Let $\Tilde{\Tilde{\N}}_x$ be the excursion law for 
super-Brownian motion in $D$ with branching mechanism $\breve\psi$.  Then for every
$\phi\ge 0$,
$$
 \tilde\N_x(1-e^{-\langle X^k,\phi\rangle})
 = \Tilde{\Tilde{\N}}_x(1-e^{-\langle X^k,\phi\rangle}).
$$
\end{lemma}

\begin{proof}
Let $g$ be the solution to
\begin{align*}
\frac12\Delta g(y)&= \breve\psi(y,g(y))\text{ for }y\in D_k \\
 g&=\phi \text{ on } \partial D_k.
\end{align*}
Then $f=g+u^N$ is the solution to  
\begin{align*}
 \frac{1}{2}\Delta f&=\psi(f) \text{ in } D_k \\
 f&=\phi+u^N \text{ on } \partial D_k,
\end{align*}
so 
\begin{align*}
&\tilde\N_x(1-e^{-\langle X^k,\phi\rangle})
 = \N_x\Big((1-e^{-\langle X^k,\phi\rangle})
     e^{-\langle X^k,u^N\rangle}\Big) \\
&\qquad= \N_x(1-e^{-\langle X^k,\phi+u^N\rangle})
    -\N_x(1-e^{-\langle X^k,u^N\rangle}) \\
&\qquad= \Big(g(x)+u^N(x)\Big)-u^N(x)=g(x)
= \Tilde{\Tilde{\N}}_x(1-\exp-\langle X^k,
  \phi\rangle).
\end{align*}
This suffices. 
\end{proof}

{\bf $\Upsilon$ Backbone:}
Under $\breve\N_x$ we begin with one particle in the system that performs a $v_N$
transform of the motion with generator $\breve L$ in $D$. Note that by Lemma \ref{harmonic}, $v_N$  is $\breve L$-superharmonic (as indeed are all the $v_A$). Let $-V_A(\cdot)$ denote the right hand side of (b) of Lemma \ref{harmonic}. 

If the particle dies at some site $y$, a random number of particles are born in its place, and numbered 1, 2, \dots, $j$. The first particle is then assigned a randomly generated tag $A_1$, the second is tagged with $A_2$, etc.  For $j \geq2$, the probability that $j\ge 2$ particles are born and that their tags are a specified sequence $A_1, \dots, A_j$ is given by:

\begin{equation} \label{prob}
\f{1}{V_N(y)\cdot j!}b(j,u^N(y),y)\prod_{i=1}^jv_{A_i}(y).
\end{equation} 
Here the $A_i$ are chosen so that $A_i\neq\emptyset$ and $\cup_{i=1}^j A_i = N$. In full generality, particles need not die before exiting $D$, but in our main example below they will in fact always die in the interior of $D$ (and the $v_A$ will accordingly be $\breve L$-potentials). Note that the above defines a probability measure, by definition of $V_N$. 

Each of the particles so created now evolves as a $v_A$
transform of the motion with generator $\breve L$ (where $A$ is the particle's label), until it dies. Whereupon a new collection of random branches is created, each labeled by an ${A'}_i\neq\emptyset$ with $\cup_i {A'}_i = A$, etc.

It will be convenient to describe this backbone as follows. Let $n_t$ denote the number of particles alive at time $t$. Label them using an index $i$, $1\le i\le n_t$, and for each one let $x_i(s)$, $0\le s\le t$ denote the location of the particle or its ancestor at time $s$. Let $\Upsilon^k_t$ be the measure putting unit mass at each point $x_i(t)$ for which $x_i$ has not exited $D_k$ by time $t$. Then $\Upsilon^k$ represents the backbone killed upon leaving $D_k$, and we recover the whole backbone by letting $k\to\infty$. Without comment we will feel free to refer to $\Upsilon^k$ in terms of the underlying particles, though formally it is still a measure-valued process.

{\bf Immigration at nodes:}
The birth of $j$ particles at $y$ is accompanied by site-dependent creation of mass. Conditional on the location being $y$ and the number of particles being $j\ge 2$, the mass created is a random variable $R_j\ge 0$ whose law $\mu_{j, y}(dr)$ is given for $r>0$ by

$$
\mu_{j,y}(dr) = 
\begin{cases}
\f{r^j e^{-ru^N(y)}\,\pi(dr)}{\int_0^\infty r^j  e^{-ru^N(y)}\,\pi(dr)}, &j\ge 3\\
\f{r^2 e^{-ru^N(y)}\,\pi(dr)}{2a_2+\int_0^\infty r^2  e^{-ru^N(y)}\,\pi(dr)}, &j=2.
\end{cases}
$$
In the case $j=2$, $\mu_{j,y}$ also has an atom at $r=0$ of size 
$2a_2/[2a_2+\int_0^\infty r^2  e^{-ru^N(y)}\,\pi(dr)]$. In other words, if $a_2>0$ then it is possible for
no mass to be created at a branch with $j=2$. 
In any case, all mass created at the node then evolves according to $\tilde \N_y$. 

Temporarily writing $Y_{j,y}^k$ for the exit measure from $D_k$ resulting from such a creation of mass, and $E_{j,y}$ for expectations under this conditional law, we can write $Y^k_{j,y}$ as $\int\nu \,N(d\nu)$ where $N$ is a Poisson random measure with intensity $n(d\nu)=R_j\tilde \N_y(X^k\in d\nu)$. Thus 
\begin{align}
E_{j,y}[e^{-\langle Y^k_{j,y},\phi\rangle}]
&= E_{j,y}[e^{-\int \langle \nu,\phi\rangle\,N(d\nu)}]
= E_{j,y}[e^{-\int(1-e^{-\langle \nu,\phi\rangle})\, n(d\nu)}]\nonumber \\
&= E_{j,y}[e^{-R_j\tilde \N_y(1-e^{-\langle X^k,\phi\rangle})}]
= \int_{[0,\infty)} e^{-r\tilde \N_y(1-e^{-\langle X^k,\phi\rangle})}\,\mu_{j,y}(dr)\label{PPP}\\
&=\frac{b(j,u^N(y)+\tilde \N_y(1-e^{-\langle X^k,\phi\rangle}),y)}{b(j,u^N(y),y)}.\nonumber
\end{align}
Denote this quantity by $M^k(j,\phi,y)$. 

{\bf Immigration along branches:}  

For any $\lambda>0$
and $y \in D$ define 
\begin{align*}
\eta (y,\lambda) &= \psi'( u^N(y) + \lambda) -
\psi'(u^N(y))\\
&=2a_2\lambda+\int_0^\infty re^{-ru^N(y)}(1-e^{-\lambda r})\,\pi(dr).
\end{align*}
Notice that this has the same form as $\psi'$ in \eqref{psi^(n)} but with $a_1=0$ and a spatially varying $\pi_y(dr)=e^{-ru^N(y)}\,\pi(dr)$. 

We create mass along the branches. The mass created to time $t$ forms a spatially dependent L\'evy process, with L\'evy exponent $\eta$. In other words, if $L_t$ is the mass created until time $t$, then
$$
\breve\N_x(e^{-\lambda L_t})=\breve\N_x(e^{-\int_0^t\int \eta(y,\lambda)\Upsilon^k_s(dy)\,ds}).
$$
This mass then evolves according to $\tilde\N_y$. In particular, if $\tau$ denotes the first branch time and $\xi_t$ the position along the first branch to time $t$ then
$$
\breve\N_x(e^{-\lambda L_\tau})=\breve\N_x(e^{-\int_0^\tau\eta(\xi_t,\lambda)\,dt}).
$$
If $Y^k_\text{Br}$ denotes the exit measure of the mass created along this branch till time $\tau$ then a calculation similar to \eqref{PPP} shows that
\begin{align}
\breve\N_x(e^{-\langle Y^k_\text{Br},\phi\rangle})
&=\breve \N_x(e^{-\int_0^\tau \tilde\N_{\xi_t}(1-e^{-\langle X^k,\phi\rangle})\,dL_t})\label{branch}\\
&=\breve \N_x(e^{-\int_0^\tau \eta(\xi_t,\tilde\N_{\xi_t}(1-e^{-\langle X^k,\phi\rangle}))\,dt}).\nonumber
\end{align}
In other words, we obtain an expression similar to that of \eqref{Nt}.

With this, we've finished describing the construction of an exit measure $Y^k$ under a probability measure $\breve \N_x$, and are ready for the following:

\begin{theorem} \label{theorem} Assume conditions \eqref{condition} and \eqref{vPositivity}.  Then
$$
\breve{\M}_x(\exp-\langle X^k, \phi\rangle ) = \breve{\N}_x(\exp-\langle Y^k, \phi\rangle ).
$$
\end{theorem}

Note that this shows that, under the two measures $\breve{\M}_x$ and $\breve{\N}_x$, the exit measures from each $D_k$ have distributions which agree. Using historical processes (as in \cite{sv1}) one can show that the same result carries over to the process level. That is, that the laws of the full superprocesses agree under these two measures. 

\begin{proof}

In the present context, it is useful to label all the particles of
$\Upsilon^k$ that exit $D_k$, by placing an order on them.
So let $F_k$ be the set of such particles, and set $\gamma_k=|F_k|$.
For $A\subset N$, let 
\begin{equation*}
\mathcal{S}_m(A)=\{(C_1,\dots,C_m) : C_1\cup\dots\cup C_m=A, 
\emptyset\neq C_i\;\forall i\}. 
\end{equation*}
If $\gamma_k=m$, choose at random an ordering of $F_k$,
and for
$\Lambda=(C_1,\dots,C_m)\in\mathcal{S}_m(N)$, write 
$\Upsilon^k\approx\Lambda$ for the event that the $i$th particle is 
tagged with the set $C_i$, $i=1,\dots,m$. Thus for example,
\begin{equation}
\label{labelling}
 \breve\N_x(\gamma_k=m)
 =\sum_{\Lambda\in\mathcal{S}_m(N)} 
\breve\N_x(\Upsilon^k\approx \Lambda).
\end{equation}

Note that if $M_1,\dots,M_j$ are disjoint ordered sets, with $|M_i|=k_i$ and $m=\sum k_i$ 
then there are $\frac{m!}{k_1!k_2!\dots k_j!}$ orderings of $S=\bigcup M_i$ which 
are compatible with the given orders on each $M_i$. In other words, if $\sigma$ is any
order on $S$, and if $\Sigma$ is an order on $S$ picked at random, then
the conditional probability
\begin{equation}
\label{random partition}
P(\Sigma=\sigma\mid\Sigma_{M_i}=\sigma_{M_i}, 
i=1,\dots,j)=\frac{k_1!k_2!\dots k_j!}{m!}
\end{equation}
(writing $\sigma_M$ etc\dots for the restriction of $\sigma$ to $M$).

As described initially, the root particle of the tree is always a
$v^N$-particle. It is convenient, for purposes of induction, to allow the
same notation to cover the situation that we start with our root
being a $v^A$-particle for some $A\subset N$. In this case, 
\eqref{labelling} still holds, but with 
$\Lambda\in\mathcal{S}_m(N)$ replaced by $\Lambda\in\mathcal{S}_m(A)$.
With this in mind, we may define another restriction operation as follows. For
$1\le i_1<\dots<i_k\le m$, set 
\begin{equation*}
(C_1,\dots,C_m)|_{\{i_1,\dots,i_k\}}=(C_{i_1},\dots,C_{i_k}).
\end{equation*}
Thus, if $\Lambda=(C_1.\dots,C_m)\in\mathcal{S}_m(A)$ and
$M\subset\{1,\dots,m\}$, we will have that $\Lambda|_M\in\mathcal{S}_m(B)$,
for $B=\cup_{i\in M} C_i$. As a shorthand for the latter, we write 
$\Lambda(M)=\cup_{i\in M} C_i$.

The case of interest is that the first branch of $\Upsilon^k$ partitions $\{1,\dots,m\}$ 
via the descent relation. If $\beta$ is this partition then there are $|\beta|$ particles
born at this branch, and 
\begin{equation}\label{beta factorial}
|\beta|! \text{  ways of tagging the 1st particle, 2nd particle, etc.}
\end{equation}
with distinct elements of $\beta$. 

We will show, by induction on $m\ge 1$, that for $\emptyset\neq
A\subset N$, and $(C_1,\dots,C_m)\in\mathcal{S}_m(A)$,
\begin{align}
 \label{induction}
 \breve\N_x(\exp - \langle Y^{k}, \phi \rangle>
 &\Upsilon^k\approx(C_1,\dots,C_m)) \\
 &=\frac{1}{m!v_A(x)}\N_x(e^k_{\phi+u^N}\prod_{i=1}^m \langle
 X^k,v_{C_i}\rangle).\nonumber
\end{align}
Taking $A=N$ and summing over $\mathcal{S}_m(N)$ will then establish the
theorem.

We start with the case $m=1$. 
Here $ \Upsilon^k$ will have a single $v_{N}$- process with lifetime $\zeta\ge \tau_{k}$. Therefore,
\begin{align}
 \label{induction1}
& \breve\N_x(\exp -\langle Y^{k}, \phi \rangle; 
 \Upsilon^k\approx N) \nonumber\\
 &= E_x^{\psi'\circ u^N,v_N} \left [ 1_{\zeta < \tau_k} e^{-\int_0^\zeta
\, \eta(\xi_s,\tilde{\N}_{\xi_s}(1- e^k_{\phi}))\,ds} \right ] \nonumber\\
 &= \frac{1}{v_N(x)} E_x^{\psi'\circ u^N} \left [ v_{N}(\xi_{\tau_{k}}) 1_{\zeta < \tau_k} e^{-\int_0^\zeta
\, \eta(\xi_s,\tilde{\N}_{\xi_s}(1- e^k_{\phi}))\,ds} \right ] \nonumber\\
 &= \frac{1}{v_N(x)} E_x \left [ v_{N}(\xi_{\tau_{k}})  e^{-\int_0^{\tau_{k}}
\, \psi'\circ u^N(\xi_{s})\,ds} e^{-\int_0^{\tau_{k}} \, \eta(\xi_s,\tilde{\N}_{\xi_s}(1- e^k_{\phi}))\,ds} \right ] \nonumber\\
 &= \frac{1}{v_N(x)} E_x \left [ v_{N}(\xi_{\tau_{k}})  \CN_{\tau_{k}}(e^k_{\phi + u^N} )  \right] \nonumber\\
 &=\frac{1}{v_N(x)}\N_x(e^k_{\phi+u^N} \langle
 X^k,v_{N}\rangle),\nonumber
\end{align}
where the second last equality follows from definition of $\eta$ and $\CN$, and the last equality follows from Lemma \ref{background} (c).

Turning to the inductive step, let $m>1$ and
assume the inductive hypothesis for all $A\subset N$, and for all values
smaller than $m$. For simplicity, we will
verify~\eqref{induction} in the case $A=N$. For $\zeta$ the
lifetime of the initial particle, and $\Lambda=(C_1,\dots,C_m)$, we have by
\eqref{branch}, \eqref{PPP}, \eqref{beta factorial}, \eqref{prob}, 
\eqref{random partition}, and the definition of $\Upsilon$,
that
\begin{align*} 
&\breve{\N}_x(e^{-\langle Y^k,\phi\rangle }, \Upsilon^k
\approx \Lambda)\\ 
&\hspace{0.5in}= \sum_{\substack{\beta \in \CP(\{1,\ldots m\})\\ |\beta|\ge 2}}
E_x^{\psi'\circ u^N,v_N} \left [ 1_{\zeta < \tau_k} e^{-\int_0^\zeta
\, \eta(\xi_s,\tilde{\N}_{\xi_s}(1- e^k_{\phi}))\,ds} \times \right .\\ 
& \hspace{1.5in}\left.  \times M(| \beta |, \phi, \xi_\zeta) \times | \beta | ! \times \f{
b(|\beta|,u^N, \xi_{\zeta}) \prod_{A \in \beta }
v_{\Lambda(A)}(\xi_\zeta)} {V_N| \beta | !}\times\right.\\
& \left . \hspace{1.5in} \times \f{ \prod_{A \in \beta } | A| !}{m!}
\times\prod_{A \in \beta } \breve{\N}_{\xi_\zeta}(
e^{-\langle Y^k, \phi \rangle }; \Upsilon^k \approx \Lambda |_A ) \right ].
\end{align*}

By \eqref{u process} and \eqref{killed process} this equals
\begin{align*}
&\f{1}{m!  v_N(x)} E_x^{\psi'\circ u^N} \left ( 
\sum_{\substack{\beta \in \CP(\{1,\ldots m\})\\ |\beta|\ge 2}}
\int_0^{\tau_k} \,e^{-\int_0^t 
\eta(\xi_s,\tilde{\N}_{\xi_s}(1- e^k_{\phi}))\,ds} \times \right .  \\
&\hspace{1.5in}
\left.\times b(|\beta|,\phi + u^N, \xi_{t}) 
\times\prod_{A \in \beta } | A| !
v_{\Lambda(A)}(\xi_t) \breve{\N}_{\xi_t}( e^{-\langle Y^k, \phi \rangle };
\Upsilon^k \approx \Lambda |_A ) \right)\,dt\\
&\qquad=
\f{1}{m! v_N(x)} E_x \left ( \sum_{\substack{\beta \in \CP(\{1,\ldots m\})\\ |\beta|\ge 2}}
\int_0^{\tau_k} \, e^{-\int_0^t [
\psi'(u^N(\xi_s)) + \eta(\xi_s,\tilde{\N}_{\xi_s}(1-
e^k_{\phi}))]\,ds} \right . \times \\
&\hspace{1.5in}
\left.\times b(|\beta|,\phi + u^N, \xi_{t}) \prod_{A \in \beta } | A| !
v_{\Lambda(A)}(\xi_t) \breve{\N}_{\xi_t}( e^{-\langle Y^k, \phi \rangle };
\Upsilon^k \approx \Lambda |_A ) \right )\,dt.
\end{align*}

By definition of $\eta$ this is
\begin{multline*}
\f{1}{m! v_N(x)}
E_x \left ( \sum_{\substack{\beta \in \CP(\{1,\ldots m\})\\ |\beta|\ge 2}}
\int_0^{\tau_k}
\,\CN_t(e^k_{\phi + u^N} ) b(|\beta|,\phi + u^N, \xi_{t}) \right
. \times \\ 
\left.\times\prod_{A \in \beta } | A| !
v_{\Lambda(A)}(\xi_t) \breve{\N}_{\xi_t}( e^{-\langle Y^k, \phi \rangle };
\Upsilon^k \approx \Lambda |_A ) \right )\,dt,
\end{multline*}
which by induction equals
\begin{multline*}
\f{1}{m!  v_N(x)}
E_x \left ( \sum_{\substack{\beta \in \CP(\{1,\ldots m\})\\ |\beta|\ge 2}}
\int_0^{\tau_k} \,
\CN_t(e^k_{\phi + u^N} ) b(|\beta|, \phi + u^N, \xi_{t})\times
\right .\\ 
\times \left . \prod_{A \in \beta }\N_{\xi_{t}} (e^k_{\phi+ u^N}
\prod_{i \in A} \langle X^k,v_{C_i}\rangle  ) \right )\,dt.
\end{multline*} 
By Lemma \ref{palmformula} we conclude that 
$$
\breve{\N}_x(e^{-\langle Y^k,\phi\rangle }, \Upsilon^k
\approx \Lambda)
=\f{1}{m! v_N(x)}\N_x (
e^k_{\phi + u^N}\prod_{i=1}^m \langle  X^k, v_{C_i}\rangle ).
$$
\end{proof}

\section{Conditioning the exit measure to hit $n$ points}\label{hittingNpoints}

In this section we shall consider the exit measure when it is
conditioned to give positive mass to $n$ small balls on the boundary
of $D$. We shall study the limit as the radius of these small balls
tends to $0$.  In the first part, we set notation. In the second part,  we discuss
the case when $\psi$ is analytic. Here the limit  is given
by a martingale change of measure as in the previous section. Then we
describe an ``explosion'' phenomenon when $n=2$ and $\psi(\lambda) =
\lambda^{1+\beta}$, $0 \leq \beta \leq 1.$ Finally we discuss instablity
in the limiting process by establishing a range of possible limits. We begin by fixing some notation for this section.

\subsection{Conditioning to hit $n$ small balls}\label{hittingNballs}

Let $D$ be a $C^2$-bounded  domain in $\R^d$, for $d \geq 4$, $N =
\{ 1,2,\ldots, n\}$ for $n \in \N$, and let $\{z_1, z_2, \ldots, z_n\}$ be
distinct points on the boundary of $D$. Let $U$ be the potential
operator for Brownian motion killed upon hitting the boundary of $D$. For $x \in D$ let
$K_x^D(\cdot,\cdot)$ be the Martin Kernel on $D$ and let $G_D(\cdot,\cdot)$ be the Green function.  For any $z \in \partial D$ and $\epsilon >0$ we define
$\Delta(z, \epsilon) = B(z,\epsilon) \cap \partial D$ where
$B(z,\epsilon)$ is the Euclidean ball of radius $\epsilon$ in $\R^d$.

For any $A \subset N,$ set $$B^A_\epsilon = \cup_{i \in A} B^i_\epsilon \equiv \cup_{i \in A} B(z_i, \epsilon) 
\mbox{ and } \Delta^A_\epsilon = \cup_{i \in A}  \Delta^{i}_\epsilon \equiv \cup_{i \in A} \Delta(z_i, \epsilon). $$
Fix $0 < \delta_0 < 1$ such that $B(z_i, \delta_0) \cap B(z_j, \delta_0) = \emptyset$ if $i \neq j \in N$ and 
for any $\epsilon >0$, set  
$$D^A_\epsilon = D \backslash \overline{\cup_{i \in A} B(z_i, \epsilon)}
\mbox{ and } \tau_\epsilon  = \tau^A_\epsilon = \tau_{D^A_\epsilon}.$$
For any $x \in D$, define the functions:
\begin{eqnarray*}
 u^A_\epsilon (x) &=&  \N_x \Big(\sum_{i \in A} \langle  X^D, 1_{\Delta(z_i, \epsilon)} \rangle >0  \Big),\\ 
 v^A_\epsilon (x) &=&  \N_x \Big(\prod_{i \in A} \langle X^D, 1_{\Delta(z_i, \epsilon)} \rangle >0  \Big),\\ 
\mbox{and}&&\\
 v_{A,\epsilon}(x) &=& \N_x (\mbox{ $X^D$ charges  all  $\Delta^i_\epsilon$ for $i \in A$ and does not charge  $\Delta^j_\epsilon$ for $j \not \in A$ }).\\ 
\end{eqnarray*}
It is easy to see by an exclusion-inclusion argument that
\begin{eqnarray} \label{uvrel2}
v_{A,\epsilon} &=&
\sum_{\substack{N \setminus A \subset B \subset N}}
(-1)^{| A | + | B | + n + 1 } u^B_\epsilon, \nonumber \\
v^A_\epsilon &=& -\sum_{\substack{ \emptyset \not = B \subset A}}
(-1)^{ | B | } u^B_\epsilon, \nonumber \\
u^A_\epsilon &=& -\sum_{\emptyset \neq B \subset A} (-1)^{|B|} v^B_\epsilon.
\end{eqnarray}
Thus (\ref{uvrel}) and (\ref{vPositivity}) hold. By (a) of Lemma \ref{background}, $\frac12\Delta u^A_\epsilon=\psi(u^A_\epsilon)$. 

We will need the following two Lemmas in section \ref{asymptotics}.

\begin{lemma} \label{basic1}  Let $D$ be a bounded $C^2$ domain in dimension $d\ge 4$. Let $B$ be $d$-dimensional \BM\/ started from $x \in D$, under a probability measure $P_x$.  For $y \in D$ let $m_y^D$ be the harmonic measure starting from $w$. Let $ z \in \partial D$ and  $z_{0} \in \partial D$.   Then
for $x\in D$ fixed, $\exists$ constants $C_k, k = 1,2,3,4$ such that 
\begin{enumerate}

\item $P_y(B_{\tau_{2 \epsilon}} \in \partial B(z, 2 \epsilon) ) \leq  C_1 m_y^D( \Delta(z, \epsilon))$ for $y \in D\setminus B(z,4\epsilon)$,\\

\item $\int_D G_D(x,y) m_y^D(\Delta(z,\epsilon))^2\,dy \leq C_2 \epsilon^2 m_x^D(\Delta(z,\epsilon))$,\\

\item $K^D_x(y,z)\le C_3|y-z|^{-d}\mbox{dist}(y,\partial D)$, and \\

\item $m_x(\Delta(z,\epsilon)\ge C_4\epsilon^{d-1}$. 
%\item $\displaystyle \lim_{\eta \rightarrow 0 } \sup_{y \in \partial B({z_{0}}, {\eta})}K^{D}_{x}(y, z) = 0$.
\end{enumerate}
\end{lemma}

\begin{proof} This uses the proof of Lemma 5.4 in \cite{sv1}, which in turn uses an argument from \cite{al}.  More specifically,  (a) follows from (5.6) in \cite{sv1} and  (b) follows from (5.8) in  \cite{sv1}. We note that in \cite{sv1} the domain $D$ is assumed to be bounded Lipschitz domain which includes the class of $C^2$-bounded domains. Part (c) follows from Lemma 2.1 in \cite{al}. Part (d) is just the fact that the density of harmonic measure with respect to surface area is bounded away from 0. 
%If $\eta >0$  is small enough then  $z \not \in \Delta (z_{0}, \eta).$ As the Martin kernel is 
%continuous in $D$, it is bounded on  $\partial B({z_{0}}, {\eta})$ and  the conclusion 
%follows from the fact that $K^{D}_{x}(\cdot, z)$ vanishes on $\partial{D} \setminus \{z\}$ .
\end{proof}

Set $\rho=(d-3)/(d-1)$, and note that $0<\rho<1$ since $d\ge 4$. 

\begin{lemma} \label{thetaest} Let $D$ be a bounded $C^2$ domain in
dimension $d\ge 4$, $A \subset N=\{1, \dots, n\}$, and $\lambda >0$. Then
there exists $
\theta>0$ and $\epsilon_0>0$,    such that
whenever $0<\epsilon<\epsilon_0$ and
  $y \in D^{A}_{ \theta \epsilon^{\rho}} $, $\forall \,\,\, i \in A $ then
   \begin{eqnarray*}
 & v^{i}_{\epsilon}(y) \leq \lambda.
\end{eqnarray*}
\end{lemma}
\proof   (a) By the maximum principle, it is enough to prove the lemma for
 $y \in D^{A}_{\theta \epsilon^{\rho}} $ and $\mbox{dist}(y, \partial D)
\leq \theta \epsilon^{\rho}.$

Fix an $i \in A$, and denote $v^{i}_{\epsilon}$ by
$v_{\epsilon}$ and $z_{i}$ by $z$.  Let  $\theta >2 ,  \epsilon_{0} <
\min{(\delta_0, 1)}, 0 < \epsilon < \epsilon_{0}$. Using the comparison
principle
(see
8.2.H in \cite{dy02}), and  (2.2.10) in \cite{ver} there exists $c_1
>0$
such that \begin{equation}\label{distest} v_\epsilon(x) \leq c_1
\mbox{
dist}(x, \partial D)^{-2}, \forall x \in D.\end{equation}
Now using the Feynman-Kac formula we have,
\begin{eqnarray}
  v_\epsilon(y)  &= & E_y \left (v_\epsilon(B(\tau_{2\epsilon})) \exp(
-\int_0^{\tau_{2\epsilon}}
\frac{\psi(v_{\epsilon}(B_s))}{v_{\epsilon}(B_s)}\,ds)
1(\tau_{2\epsilon} <
\tau_{D}) \right) \nonumber \\
                &\leq & E_y \left (v_\epsilon(B(\tau_{2\epsilon}));
1(\tau_{2\epsilon} < \tau_{D} ) \right )\nonumber \\
                &\leq & c_1\epsilon^{-2} P_y( \tau_{2\epsilon}  \leq
\tau_{D} ) \nonumber \\
& \leq & c_{2 }\epsilon^{-2}P_y( B_{\tau_{D}} \in
\Delta(z,
 \epsilon) ),    \label{fkest}
\end{eqnarray}
where the last inequality follows from the proof of Theorem 3.1 in
\cite{al}.  As $\rho <1 $ and $\theta >2$,  $ 0 < \epsilon < \frac{1}{2}
|y-z|$. Using the above and Lemma 2.1 in \cite{al}, we have that
\begin{equation} 
v_\epsilon(y)  \leq c_{3} |y-z|^{-d} \mbox{dist}( y, \partial D)
\label{ALGbound}
\epsilon^{d-3}. 
\end{equation}
 Since   $D$ is bounded we
have
for all  $y \in D^{\{i\}}_{\theta \epsilon^{\rho}}$,
\[  v_\epsilon(y)  \leq c_{3} (\theta\epsilon^{\rho} )^{-d
+1}\epsilon^{d-3} \leq c_{3} \theta^{-d +1} \epsilon^{d-3 + (- d +1)\rho}
=  c_{3} \theta^{-d+1}. 
\]
As $ d \geq 4$, we can choose $\theta$ large enough to obtain the required
upper bound. The argument is independent of $i$ and  the proof of the
lemma is  complete.
\qed

\subsection{Asymptotics for analytic $\psi$}\label{asymptotics}

In this section we shall assume that $\psi$ is a real analytic
function. We also assume that $a_1=0$ (a version of the results should hold for $a_1>0$ as well, but for the Martin kernel of killed Brownian motion rather than of Brownian motion itself).  Note that if $a_1=0$ then (A1) $\Rightarrow$  \eqref{condition}. The results presented here and their proofs  largely mirror those presented in \cite{sv1} which considered the case $\psi(u) = 2u^2$. The principal difference lies in the proof of Lemma \ref{bound}, where a more delicate argument is required in order to obtain convergence of the power series arising there. 

\begin{theorem} \label{h1h2}  Let $D$ be a bounded $C^{2}$ domain in dimension $d \geq 4$. Assume (A1) and that $a_1=0$. 
Let $A \subset N=\{1, \dots, n\}$. Then for $x,y \in D$,
$$\lim_{\epsilon \rightarrow 0 } \frac{ v^A_\epsilon(y) }{ \prod_{i \in A} v^{i}_\epsilon (x) }= K^A_x(y),$$
where 
\[ K^A_x(y) = \left \{ \begin{array}{ll} K_x^D(y,z_i), & A=\{i\}, \\
&\\
                                        \sum_{\sigma \in \CP(A)} (-1)^{|\sigma|}\psi^{(|\sigma|)}(0)U\Big(\prod_{C\in\sigma} K^{C}_x\Big)(y), & |A | \geq 2.
                       \end{array} \right. \]
\end{theorem}

As part of the argument, it will emerge that $K^A_x$ is actually finite. Note that if $|A|\ge 2$ we may also write
\begin{equation}
K^A_x(y) =  \sum_{j=2}^{|A|} \f{(-1)^j \psi^{(j)}(0)}{j!} \sum_{\substack{\cup_{i = 1}^j C_i = A, \\ C_i \neq \emptyset \text{ and disjoint} }} U(\prod_{i=1}^j K^{C_i}_x)(y).
\label{alternateindexing}
\end{equation}
It is the latter form that comes naturally out of the formulae of Sections \ref{preliminaries} and 
\ref{BrParticleDescrip}.

We need several  lemmas before starting the proof. 

\begin{lemma} \label{bound}
Assume the conditions of Theorem \ref{h1h2},  and let $\rho$, $\epsilon_0$, and $\theta$ be as in Lemma \ref{thetaest} with $\lambda<\lambda_0/2^n$. Then there is a $C < \infty$  such that $\forall 0<\epsilon < \epsilon_0$ and $y \in D^A_{2^{|A|}\theta\epsilon^{\rho}}$
\[ \frac{ v^A_\epsilon(y) }{ \prod_{i \in A} v^{i}_\epsilon (x) }  \leq C\sum_{i \in A} K^D_x(y,z_i). 
\]
\end{lemma}

\proof  We prove the lemma by induction on the size of $A$.  

{\bf Step 1: ($|A | = 1$)}   Let $A=\{z\}$  and $v_\epsilon= v^{A}_\epsilon$.  As in the proof of Lemma \ref{thetaest}, we use (\ref{distest}), the Feynman Kac formula,  and (a) of Lemma \ref{basic1}, to get
\begin{eqnarray} \label{lbound}
 v_\epsilon(y)  &= & E_y \left (v_\epsilon(B(\tau_{2 \epsilon})) \exp( -\int_0^{\tau_{2\epsilon}} \frac{\psi(v_{\epsilon}(B_s))}{v_{\epsilon}(B_s)}\,ds) \right) \nonumber \\
               &\leq & E_y \left (v_\epsilon(B(\tau_{2 \epsilon})); B(\tau_{2\epsilon} )\in \partial B(z,2 \epsilon) \right ) \nonumber \\
               &\leq & c_1\epsilon^{-2} P_y( B(\tau_{2\epsilon} )\in \partial B(z,2 \epsilon) ) \nonumber \\
                          &\leq & c_2\epsilon^{-2} m_y^D( \Delta(z, \epsilon) )
\end{eqnarray}
for $y\in D^A_{4\epsilon}$. 
Using the Palm formula (Lemma \ref{palmformula}) with $\phi = 0$ and $n=2$ we then obtain that 
\begin{equation} \label{jens}
 \N_x(X^D(\Delta(z,\epsilon))^2) = \psi^{(2)}(0) E_x \int_0^{\tau_D} [\N_{B_s}(X^D(\Delta(z,\epsilon)))]^2\, ds.  
\end{equation} 
Using the above, the Cauchy-Schwartz inequality, Lemma \ref{palmformula}, and Lemma \ref{basic1} (b)  we have
\begin{eqnarray}\label{ubound}
  v_\epsilon(x) &=& \N_x( X^D(\Delta(z,\epsilon)) >0 ) \nonumber \\
  &\geq& \frac{\N_x( X^D(\Delta(z,\epsilon)) )^2}{\N_x( X^D(\Delta(z,\epsilon))^2 )} \nonumber \\
&=& \frac{ m_x^D(\Delta(z,\epsilon))^2}{\psi^{(2)}(0)\int_D G_D(x,y) m_y^D(\Delta(z,\epsilon))^2 ) \,dy}\nonumber\\
&\geq& c_3 \epsilon^{-2} m_x^D(\Delta(z,\epsilon)).
\end{eqnarray}
 Consequently,
\eqref{lbound}, \eqref{ubound}, and the boundary Harnack principle (see \cite{bbb}) yield
\[ \f{v_\epsilon(y)}{v_\epsilon(x)} \leq c_4 \f{\epsilon^2 m_y^D(\Delta(z,\epsilon))}{\epsilon^2 m_x^D(\Delta(z,\epsilon))} \leq c_5 K_x(y,z) \]
for $y\in D^A_{4\epsilon}$. By decreasing $\epsilon_0$ if necessary, we can (and will) assume that $D^A_{4\epsilon}\supset D^A_{2\theta\epsilon^\rho}$. This establishes the case $|A|=1$.
 
{\bf Step 2: ($|A| >1 $)}  Assume the result for every proper subset of $A$. Set $q=|A|$ and $\alpha(k)=2^k\theta\epsilon^\rho$ for $k=1,\dots,q$. By hypothesis, each $v^i_\epsilon(y)<\lambda_0/2^n$ on $D^A_{\alpha(1)}$. The same is then true of $v^B$ for $B\subset A$, so by Lemma \ref{harmonic} $u^A<\lambda_0$ on $D^A_{\alpha(1)}$. Thus (\ref{v^A}) applies and by the Feynman-Kac formula, we obtain the following on $D^A_{\alpha(q)}$:
\begin{eqnarray}\label{FKac}
v_\epsilon^A(y) &=& E_y (e^{-\int_0^{\tau_{\alpha(q-1)}} \phi_\epsilon^A(B_r)\,dr } v^A_\epsilon(B_{\tau_{\alpha(q-1)}}) ) \nonumber \\
                &+&(-1)^{|A|} \sum_{j=2}^\infty \f{(-1)^j\psi^{(j)}(0)}{j!} \sum_{\substack{\cup_{i=1}^jC_i = A\\ \emptyset \neq C_i \neq A}} E_y \int_0^{\tau_{\alpha(q-1)}}  \prod_{i=1}^j v^{C_i}_\epsilon(B_t) (-1)^{|C_i|} e^{-\int_0^t \phi^A_\epsilon(B_r)\,dr }\,dt \nonumber\\
	&=& I + II.
\end{eqnarray}
Here $0 \leq \phi^A_\epsilon  = \phi( v^A_\epsilon, u^A_\epsilon).$
Consider the first term in \eqref{FKac}:
\begin{eqnarray*} I & \leq& E_y ( v^A_\epsilon(B_{\tau_{\alpha(q-1)}}) ) \\
&\leq & \sum_{i \in A} E_y (v^A_\epsilon(B_{\tau_{\alpha(q-1)}}) 1(B_{\tau_{\alpha(q-1)}} \in \partial B(z_i,\alpha(q-1))))\\
&\leq& \sum_{i \in A} \sup_{ \partial B(z_i, \alpha(q-1))} v_\epsilon^A (\cdot) \,\, P_y( B_{\tau_{\alpha(q-1)}} \in \partial B(z_i,\alpha(q-1))) \\
&\leq& c_6 \sum_{i \in A} \sup_{ \partial B(z_i, \alpha(q-1))} v_\epsilon^A (\cdot) \,\, m_y^D(\Delta(z_i,\alpha(q-2))),
\end{eqnarray*}
where the last inequality follows from Lemma \ref{basic1} (a) [with $z_i$ for $z$ and $\alpha(q-2)$ for $\epsilon$]. We know that $v^A_\epsilon \leq v_\epsilon^{A\backslash\{i\}}$. Using this and \eqref{ubound} with $z=z_i$, we have
\begin{eqnarray} 
\f{I}{ \prod_{j \in A} v^j_\epsilon(x)} 
&\leq & c_6\sum_{i \in A} \f{\sup_{ \partial B(z_i, \alpha(q-1))}v^{A\backslash\{i\}}_\epsilon(\cdot)  }{ \prod_{j \in A, j \neq i} v^j_\epsilon(x)} \frac{m_y^{D}(\Delta(z_i, \alpha(q-2)))}{v^i_\epsilon(x)} \nonumber \\
&\leq & c_7 \sum_{i \in A} \f{\sup_{ \partial B(z_i, \alpha(q-1))}v^{A\backslash\{i\}}_\epsilon(\cdot)  }{ \prod_{j \in A, j \neq i} v^j_\epsilon(x)} \frac{m_y^{D}(\Delta(z_i, \alpha(q-2)))}{\epsilon^{-2} m_x^D(\Delta(z_i, \epsilon))} \label{firsttermbound}. 
\end{eqnarray} 
By the boundary Harnack principle, 
$$
m_y^{D}(\Delta(z_i, \alpha(q-2)))\le c_8 m_x^{D}(\Delta(z_i, \alpha(q-2)))K_x(y,z_i),
$$
and as in (\ref{ALGbound}), $m_x^{D}(\Delta(z_i, \alpha(q-2)))\le c_9\alpha(q-2)^{d-1}\le c_{10} \epsilon^{d-3}$. Likewise $\epsilon^{-2} m_x^D(\Delta(z_i, \epsilon))\ge c_{11}\epsilon^{d-3}$, by (d) of Lemma \ref{basic1}.  By induction, $v^{A\backslash\{i\}}_\epsilon(\cdot) / \prod_{j \in A, j \neq i} v^j_\epsilon(x)$ is bounded, in particular, on $ D^{A\setminus\{i\}}_{\alpha(q-1)}$ by $c_{12}\sum_{j\in A\setminus\{i\}}K(\cdot, z_j)$, which by Lemma \ref{basic1} (c) is bounded by $c_{13}\alpha(q-1)$ on $\partial B(z_i, \alpha(q-1))$. Therefore
\begin{equation}
\frac{I}{\prod_{j \in A}v^j_\epsilon(x)}\le  c_{14}\alpha(q-1)
%\sum_{i\in A} \left(\sup_{u \in \partial B(z_i, \alpha(q-1))}  K(u, z_j) \right) K_x(y,z_i)  \leq c_{10} \sum_{i\in A}  K_x(y,z_i) \label{term1}
\sum_{i\in A}  K_x(y,z_i) \label{term1}
\end{equation}
Let $$E_j = \{ \{C_i\}_{i=1}^j :  \emptyset \neq C_i \neq A, \cup_{i=1}^j C_i = A, \mbox{ and }  |A| + \sum_{i=1}^j |C_i| \mbox{ is even} \}$$ and 
$$O_j = \{ \{C_i\}_{i=1}^j :  \emptyset \neq C_i \neq A, \cup_{i=1}^jC_i = A, \mbox{ and }  |A| + \sum_{i=1}^j |C_i| \mbox{ is odd} \}.$$

Then the second term (II) in \eqref{FKac} is, 
\begin{eqnarray*}
 & = &\sum_{j=2}^\infty \f{(-1)^j\psi^{(j)}(0)}{j!} \sum_{E_j} E_y \int_0^{\tau_{\alpha(q-1)}} \prod_{i=1}^j v^{C_i}_\epsilon(B_t)  e^{-\int_0^t \, \phi^A_\epsilon(B_r) \,dr}\,dt \nonumber \\&& - \sum_{j=2}^\infty \f{(-1)^j\psi^{(j)}(0)}{j!} \sum_{O_j} E_y \int_0^{\tau_{\alpha(q-1)}}  \prod_{i=1}^j v^{C_i}_\epsilon(B_t)  e^{-\int_0^t \,\phi^A_\epsilon(B_r)\,dr }\,dt \nonumber \\
&\leq& \sum_{j=2}^\infty \f{(-1)^j\psi^{(j)}(0)}{j!} \sum_{E_j} E_y \int_0^{\tau_{\alpha(q-1)}} \prod_{i=1}^j v^{C_i}_\epsilon (B_t)  \,dt
\end{eqnarray*}
as all terms in the second summand are non-negative.
Therefore
\begin{eqnarray} 
\frac{II}{\prod_{i\in A} v^{i}_\epsilon (x)}  &\leq& \sum_{j=2}^\infty \f{(-1)^j\psi^{(j)}(0)}{j!} \sum_{E_j} E_y \int_0^{\tau_{\alpha(q-1)}}  \frac{\prod_{i=1}^j v^{C_i}_\epsilon (B_t)}{\prod_{i\in A} v^{i}_\epsilon (x)} \,dt.  \label{term20}
\end{eqnarray}
First consider the case  $j \leq |A|$ and $\{C_i\}_{i=1}^j \in E_j$. We observe that for $k \neq l$, if $C_k\setminus C_l\neq \emptyset$ then
$$v^{C_k}_\epsilon v^{C_l}_\epsilon \leq v^{C_k\setminus C_l}_\epsilon v^{C_l}_\epsilon.$$
So every term in this sum with $j \leq |A|$ is dominated by another term in which  $\{C_i\}_{i=1}^j \in E_j$ are such that $C_k \cap C_l = \emptyset$ for $k \neq l$. Thus it suffices to bound such terms. For disjoint $\{C_i\}_{i=1}^j \in E_j$,
\begin{eqnarray} \label{term22}
 E_y \int_0^{\tau_{\alpha(q-1)}} \f{\prod_{i=1}^j v^{C_i}_\epsilon (B_t)}{\prod_{i\in A} v^{i}_\epsilon (x)}\,dt &\leq& C^j  E_y \int_0^{\tau_{\alpha(q-1)}}  \f{\prod_{i=1}^j   \sum_{k \in C_i} K_x(B_t, z_k)\prod_{l \in C_i} v^{l}_\epsilon (x) }{\prod_{i\in A} v^{i}_\epsilon (x)} \,dt\nonumber \\
&=&
C^j   E_y \int_0^{\tau_{\alpha(q-1)}}   \prod_{i=1}^j  \sum_{k \in C_i}K_x(B_t, z_k)\,dt.
\end{eqnarray}
Let $k_1, \dots, k_j$ be distinct. Then 
\begin{eqnarray} \label{term23}
E_y \int_0^{\tau_{\alpha(q-1)}} &\prod_{i=1}^j&  K_x(B_t, z_{k_i})\,dt \le G_D(\prod_{i=1}^j K_x(\cdot, z_{k_i})) (y) \nonumber \\
&\leq &  c_{15} \sum_{i=1}^jG_D( K_x(\cdot, z_{k_i})) (y) \leq c_{15} \sum_{i=1}^j K_x(y, z_{k_i}) \nonumber\\
&\leq& c_{16} \sum_{k=1}^{|A|} K_x(y, z_{k})
\end{eqnarray}
where  the second inequality is due to the fact that $z_i$'s are separated by $\delta_0$ and the third inequality follows from the ``3-G'' theorem (see (5.17) \cite{sv1}).  It follows from (\ref{term22}) that 
\begin{equation}
E_y \int_0^{\tau_{\alpha(q-1)}}  \frac{\prod_{i=1}^j v^{C_i}_\epsilon (B_t)}{\prod_{i\in A} v^{i}_\epsilon (x)}   \,dt
\le c_{17} \sum_{k=1}^{|A|} K_x(y, z_{k}).\label{term24}
\end{equation}

Now consider the case $j>|A|$. For $\{C_i\}_{i=1}^j \in E_j$ we can select $|A|$ of the $C_i$ whose union is $A$, and apply the above bound to them. For the other $C_i$ we apply Lemma \ref{thetaest}, which gives that $v^{C_i}_\epsilon\le \max_k v^k_\epsilon\le \lambda$, where $2^{|A|}\lambda<\lambda_0$. 
Using \eqref{term20}, and \eqref{term24} we have
\begin{align*}
\frac{II}{\prod_{i\in A} v^{i}_\epsilon (x)} 
&\leq c_{17} \left [  \sum_{j=2}^{|A|} \f{(-1)^j\psi^{(j)}(0) |E_j |}{j!}  +  \sum_{j= |A| +1}^{\infty} \f{(-1)^j\psi^{(j)}(0) |E_j| \lambda^{j-|A|}}{j!}   \right]   \sum_{k=1}^{|A|} K_x(y, z_k)\\
&\leq c_{17} \left [  \sum_{j=2}^{|A|} \f{(-1)^j\psi^{(j)}(0) 2^{j|A|}}{j!}  +  \sum_{j= |A| +1}^{\infty} \f{(-1)^j\psi^{(j)}(0) 2^{j |A|}\lambda^{j-|A|}}{j!}   \right] \sum_{k=1}^{|A|} K_x(y, z_k)\\ 
&\leq c_{18} \sum_{k=1}^{|A|} K_x(y, z_k).
\end{align*}
\qed

\begin{lemma} \label{unipr} Assume the conditions of Theorem \ref{h1h2}. 
Let $y \in D^A_\delta $, where $\delta<\delta_0$. Then uniformly in $z \in \partial D$,
$$\lim_{\epsilon \rightarrow 0 } P_{yz}( \exp(-\int_0^{\tau_\delta} \phi^A_\epsilon(B_t) \,dt)) = 1.$$
\end{lemma}

\proof  We have $v^{i}_\epsilon(\cdot) \rightarrow 0$ as $\epsilon \rightarrow 0.$ So by the assumptions on $\psi$, and Lemma \ref{bound} and  
$$\phi^A_\epsilon (\cdot) = \f{\psi(u_\epsilon^A + (-1)^{|A|}v_\epsilon^A) - \psi(u_\epsilon^A)}{(-1)^{|A|} v_\epsilon^A} (\cdot) \rightarrow 0.$$ 
Consequently it suffices to prove that 
$$\lim_{\lambda \rightarrow 0 } P_{yz}( \exp(-\lambda \tau_{\delta} ) )= 1,$$ uniformly in $z$. As
$D$ is Lipschitz, we have $$ \sup_{ y \in D_{\delta}, z \in \partial D} P_{yz} (\tau_\delta) < \infty$$
by \cite{cfz}. Then Lemma 3.7 in \cite{cz} implies the result.
\qed

{\em Proof of Theorem \ref{h1h2}:}  We will use induction on the size of $A$. First consider the case where our target is a single point. Let $A= \{i\}$.   When convenient in the proof we will use $D_{\delta}$ for $D^{A}_{\delta}$ for some $\delta >0.$ Let $x,y \in D_{\delta_0}$. By the Feyman-Kac Formula, for each fixed $\delta < \delta_0$ we have

\begin{eqnarray*}
 \f{v^i_\epsilon(y)}{v^i_\epsilon(x)} &=& \f{E_y(v^i_\epsilon(B_{\tau_\delta}) \exp(\int_0^{\tau_\delta} \phi^A_\epsilon(B_r)\,dr) )}{ E_x(v^i_\epsilon(B_{\tau_\delta}) \exp(\int_0^{\tau_\delta} \phi^A_\epsilon(B_r)\,dr)} \\ &=& \f{\int_{\partial D_\delta}  P_{yz}( \exp(-\int_0^{\tau_\delta} \phi^A_\epsilon(B_t) \,dt)) K_x^{D_\delta}(y,z)v^i_\epsilon(z) \,m_x^{D_\delta}(dz) }{\int_{\partial D_\delta}  P_{xz}( \exp(-\int_0^{\tau_\delta} \phi^A_\epsilon(B_t) \,dt)) v^i_\epsilon(z) \,m_x^{D_\delta}(dz) } \\
&=& \f{\int_{\partial D_\delta}  P_{yz}( \exp(-\int_0^{\tau_\delta} \phi^A_\epsilon(B_t) \,dt)) K_x^{D_\delta}(y,z)v^i_\epsilon(z) \,m_x^{D_\delta}(dz) }{\int_{D_\delta} v^i_\epsilon(z)\,m_x^{D_\delta}(dz)}  \times\\
&& \hspace{0.5in} \times \f{\int_{D_\delta} v^i_\epsilon(z)\,m_x^{D_\delta}(dz)}{\int_{\partial D_\delta}  P_{xz}( \exp(-\int_0^{\tau_\delta} \phi^A_\epsilon(B_t) \,dt)) v^i_\epsilon(z) \,m_x^{D_\delta}(dz) }. \\
\end{eqnarray*}
The measure 
\[ \lambda_{\epsilon, \delta}(x, dz) = \f{v^i_\epsilon(z)\,m_x^{D_\delta}(dz)}{\int_{D_\delta} v^i_\epsilon(z)\,m_x^{D_\delta}(dz)}\]
is a probability measure on  $\partial D^A_\delta.$ Since  the boundary of  $D^A_\delta$ is compact, by Prohorov's theorem any 
sequence $\epsilon_j$ has a subsequence, again written $\epsilon_j$, for which $\lambda_{\epsilon_j, \delta}(x, dz) \Rightarrow \lambda_{\delta}(x, dz)$
weakly in the space of probability measures. Also, $K_x^{D_\delta}(y,z)$ is continuous and bounded in $z$, for $z \in D \cap \partial D_\delta^A$, when $x,y \in D^A_{\delta_0}.$ Consequently,  Lemma \ref{unipr} implies for $x,y \in D^A_{\delta_0}$ and for all $\delta < \delta_0$
\[ \lim_{j \rightarrow \infty} \frac{v^i_{\epsilon_j}(y)}{v^i_{\epsilon_j}(x)} = \int_{\partial D_\delta} K_x^{D_\delta}(y,z) \,\lambda_\delta(x, dz). \]
The limiting function is harmonic in $y$ for $y \in D_{\delta_0}$. By a diagonlization argument, we can assume there exists a convergent subsequence of our sequence
such that the convergence  holds simultaneously  for a sequence of $\delta_j$'s  which converge to $0$. By  Lemma \ref{bound} we see then that the limit  is harmonic
in $y$ with boundary value $0$ on $\partial D \cap \partial D^A_\delta$ for all $\delta >0,$ and is $1$ at $y =x. $ This implies that  the limit is the Martin Kernel for Brownian motion in $D$. This all subsequences  have a subsequence  which converges to the Martin kernel, and so the limit itself exists.

To prove the induction step, fix $A$ and assume that the result is true for  all proper subsets of $A$.  Therefore if $\cup_{i=1}^j C_i = A$ and $\emptyset \neq C_i \neq A$ we have
\[  \lim_{\epsilon \rightarrow 0} \f{ \prod_{i=1}^j v^{C_i}_\epsilon(y)}{ \prod_{i\in  A} v^{i}_\epsilon(x) }=  \left \{ \begin{array}{ll} \prod_{i=1}^j K_x^{C_i}(y) 1_{\{C_k \cap C_l = \emptyset,\,\, 1 \leq  k \neq l \leq n\} } & \mbox{ if } 1 \leq j \leq |A| \\
 &\\
                                                          0 & \mbox{ otherwise. }       
\end{array} \right. 
\]
Let $x, y \in D$.  Let $\epsilon >0, \eta = \theta \epsilon^{\rho} >0$ be small enough so that $x,y \in D_{\eta}$ and  Lemma \ref{bound}, Lemma \ref{unipr},  and Lemma \ref{thetaest} apply. By the Feynmann Kac formula and \eqref{v^A}, we have
 then 
\begin{eqnarray}
\f{v^A_\epsilon(y)}{\prod_{i \in A}v^i_\epsilon(x) } 
&\geq &   \sum_{j=2}^\infty \f{(-1)^j\psi^{(j)}(0)}{j!}  \sum_{E_j} E_y \int_0^{\tau_{\eta}} \f{\prod_{i=1}^j v^{C_i}_\epsilon(B_t)} {\prod_{i \in A}v^i_\epsilon(x) } e^{-\int_0^t \phi^A_\epsilon(B_r)\,dr }\,dt \nonumber  \\&& - \sum_{j=2}^\infty \f{(-1)^j\psi^{(j)}(0)}{j!} \sum_{O_j} E_y \int_0^{\tau_\eta} \f{\prod_{i=1}^j v^{C_i}_\epsilon(B_t)}{{\prod_{i \in A}v^i_\epsilon(x) }}  e^{-\int_0^t \phi^A_\epsilon(B_r) \,dr}\,dt. \nonumber 
\end{eqnarray}
By  Lemma \ref{bound}, Lemma \ref{unipr},  the dominated convergence theorem (which applies as in Lemma \ref{bound} by the bound provided by Lemma \ref{thetaest}) 
 and the induction hypothesis we have
\begin{eqnarray} \label{lbound2}
\liminf_{\epsilon \rightarrow 0} \f{v^A_\epsilon(y)}{\prod_{i \in A}v^i_\epsilon(x) }
&\geq& \sum_{j=2}^{|A|} \f{(-1)^j\psi^{(j)}(0)}{j!}  \sum_{\substack{\cup_{i = 1}^j C_i = A \\ C_i \neq \emptyset = C_k \cap C_l  }} E_y \int_0^{\tau_D}  \prod_{i=1}^jK_x^{C_i}(B_t) \,dt .
\end{eqnarray}
For the upper bound,  as  in the proof of Lemma \ref{bound} we have\begin{eqnarray*}
\f{v^A_\epsilon(y)}{\prod_{i \in A}v^i_\epsilon(x) } &\leq & \f{E_y(v^A_\epsilon(B_{\tau_\eta}))}{\prod_{i \in A}v^i_\epsilon(x) } + \sum_{j=2}^\infty \f{(-1)^j\psi^{(j)}(0)}{j!}  \sum_{E_j} E_y \int_0^{\tau_\eta} \f{\prod_{i=1}^j v^{C_i}_\epsilon(B_t)} {\prod_{i \in A}v^i_\epsilon(x) }\,dt\nonumber\\
&\leq & c_4 2^{\mid A \mid -1}\eta\sum_{i \in A} 
%\left (\sup_{u \in \partial B(z_i, \eta)}  K(u, z_j) \right) 
K_x(y, z_i) + \sum_{j=2}^\infty \f{(-1)^j\psi^{(j)}(0)}{j!}  \sum_{E_j} E_y \int_0^{\tau_\eta} \f{\prod_{i=1}^j v^{C_i}_\epsilon(B_t)} {\prod_{i \in A}v^i_\epsilon(x) }\,dt.\nonumber
\end{eqnarray*}
Letting $\epsilon\to 0$, the first term $\to 0$ (since $\eta\to 0$). Using the inductive hypothesis and dominated convergence theorem  for the second term  we obtain
\begin{eqnarray} \label{ubound2}
\limsup_{\epsilon \rightarrow 0} \f{v^A_\epsilon(y)}{\prod_{i \in A}v^i_\epsilon(x) }
&\leq & \sum_{j=2}^{|A|} \f{(-1)^j\psi^{(j)}(0)}{j!}  \sum_{\substack{\cup_{i = 1}^j C_i = A \\ C_i \neq \emptyset = C_k \cap C_l  }} E_y \int_0^{\tau_D}  \prod_{i=1}^jK_x^{C_i}(B_t) \,dt. 
\end{eqnarray}
\eqref{alternateindexing}, \eqref{lbound2} and \eqref{ubound2} now yield the result. \qed

\begin{theorem} \label{limitingXtransform}
Assume the conditions of Theorem \ref{h1h2}. Let
  $\Phi_k \in {{\mathcal F}}_k$ be bounded, and fix $x \in D$. Then
  for $ y \in D$,
\[ \lim_{\epsilon \rightarrow 0} \N_{y}(\Phi_k  | \prod_{i=1}^n \langle X^D, 1_{\Delta^i_\epsilon} \rangle > 0 ) = \f{1}{K^N_x(y)} \N_y(\Phi_k M_k^N), \]
where $$
M_k^N = \sum_{\sigma \in \CP(N)} \prod_{C \in \sigma} \langle
X^k, K_x^C \rangle.$$
\end{theorem}

\begin{proof} We will need two preliminary Lemmas. 
\begin{lemma}\label{we} Assume the conditions of Theorem \ref{h1h2}. 
Set $W^C_\epsilon = \exp( (-1)^{|C|} \langle X^k, v^C_\epsilon \rangle) -1.$ Then
\begin{eqnarray*}
\N_y\left( \prod_{i=1}^n \langle X^D, 1_{\Delta^i_\epsilon} \rangle >0  | \CF_k \right )  &=& \sum_{j=1}^{2^{|N|} -1}
\frac{1}{j!}  \sum_{\substack{C_1 \cup C_2 \ldots C_j\\ C_1 \cup C_2 \ldots \cup C_j = N \\ \emptyset \neq C_i \forall i }}
 \left (\prod_{i=1}^j |W^{C_i}_\epsilon| \right) (-1)^{n + \sum_{i=1}^j | C_i | }
\end{eqnarray*}
\end{lemma}

\begin{proof} Using the arguments presented in \cite{dy1} or otherwise one can
verify that 
\[ 
u^{A, \lambda}_\epsilon(y)  = \N_y(1 - \exp(-\lambda \langle X^D, \sum_{i \in A}1_{\Delta_\epsilon^i} \rangle)) 
\]
increases to $u^A_\epsilon(y)$ for all $y \in D$ as $\lambda \rightarrow \infty$. From here on,
the  proof of this lemma is the same as the proof of Lemma 5.8 in \cite{sv1}. We will not present 
it again here, other than to remark that it is based on the Markov property of exit measures. Note that a different indexing system is used in \cite{sv1}, which accounts for the $j!$ factor. 
\end{proof}

\begin{lemma} \label{we2} Assume the conditions of Theorem \ref{h1h2}. 
Let $\Phi_k \in \CF_k$ be bounded. Let $C_1, \ldots, C_j$ be distinct and nonempty, with
$\cup_{i=1}^j C_i = N$. Then
\[ \lim_{\epsilon \rightarrow 0}\f{\N_y( \Phi_k \prod_{i=1}^j | W^{C_i}_\epsilon |)}{ \prod_{i=1}^n v^i_\epsilon(x) } = \N_x( \Phi_k \prod_{i=1}^j \langle X^k, K^{C_i}_x \rangle)1_{\{C_i \mbox{ disjoint }\}}. \]
\end{lemma}

\begin{proof} The proof of this lemma can be obtained by imitating the proof
of Lemma 5.9 in \cite{sv1}. The only changes one has to make are to use:
Theorem \ref{h1h2} in place of Theorem 5.3; Lemma \ref{bound} instead
of Lemma 5.4; Lemma \ref{expmom} instead of Lemma 2.7.
\end{proof}

To complete the proof of the Theorem, observe that
\begin{eqnarray*}
\lefteqn{\N_y (\Phi_k | \prod_{i=1}^n \langle X^D, 1_{\Delta^i_\epsilon} \rangle >0 )}\\
&=&\f{\N_y (\Phi_k , \prod_{i=1}^n \langle X^D, 1_{\Delta^i_\epsilon} \rangle >0 )}{\N_y (\prod_{i=1}^n \langle X^D, 1_{\Delta^i_\epsilon} \rangle >0 )}\\
&=&\f{\N_y (\Phi_k \N_y (\prod_{i=1}^n \langle X^D, 1_{\Delta^i_\epsilon} \rangle >0 | \CF_k ))}{\prod_{i=1}^n v^i_{\epsilon}(x)}\times \f{\prod_{i=1}^n v^i_{\epsilon}(x)}{v_\epsilon^N(y)}.
\end{eqnarray*}
By Theorem \ref{h1h2}, $$\f{\prod_{i=1}^n v^i_{\epsilon}(x)}{v_\epsilon^N(y)} \rightarrow \f{1}{K_x^N(y)}.$$  By Lemma \ref{we} and Lemma \ref{we2}, 
\begin{eqnarray*}
\lefteqn{\f{\N_y (\Phi_k \N_y (\prod_{i=1}^n \langle X^D, 1_{\Delta^i_\epsilon} \rangle >0 | \CF_k ))}{\prod_{i=1}^n v^i_{\epsilon}(x)} }\\
&=& \sum_{j=1}^{2^{|N|} -1}   \frac{1}{j!}\sum_{\substack{C_1 \cup C_2 \ldots C_j\\ C_1 \cup C_2 \ldots \cup C_j = N \\ \emptyset \neq C_i \forall i }}
 \f{\N_y (\Phi_k \prod_{i=1}^j |W^{C_i}_\epsilon| (-1)^{n + \sum_{i=1}^j | C_i | })}{\prod_{i=1}^n v^i_{\epsilon}(x)}\\
 \end{eqnarray*}
\begin{eqnarray*}
&\rightarrow& \sum_{j=1}^{2^{|N|} -1}  \frac{1}{j!} \sum_{\substack{C_1 \cup C_2 \ldots C_j\\ C_1 \cup C_2 \ldots \cup C_j = N \\ \emptyset \neq C_i \forall i }}
\N_y (\Phi_k \prod_{i=1}^j \langle X^k, K_x^{C_i} \rangle 1_{\{ C_i \mbox{ disjoint }\}})\\
&=& \sum_{\sigma \in \CP(A)} \N_y\Big(\Phi_k\prod_{C\in\sigma} \langle X^k,K^{C}_x\rangle\Big)=\N_y(\Phi_k M^N_k).
 \end{eqnarray*}
The statement of the Theorem then follows easily.
\end{proof}

\subsubsection{Weakening Hypothesis (A1)        }
It should be possible to weaken the assumption (A1) in Theorem \ref{h1h2} to the 
following:
$$ \mbox{ (A2)  \hspace{1in} $\int_1^\infty  r^{n}\,\pi(dr) <\infty$,} $$
when $N= \{1, 2, \ldots, n\}.$ We are able to prove Theorem \ref{h1h2} 
under this weakened assumption, though only in the case $n=2$.   We begin with a lemma.

\begin{lemma} \label{a2plemma} Let $D$ be a bounded $C^{2}$ domain in dimension $d \geq 4$. Assume (A2) and that $a_1=0$. 
Let  $ \epsilon >0.$ There exists $a,b \in 
\R, c_{1} >0$, $f_{\epsilon}, g_
{\epsilon}, h_{\epsilon}: D^{1,2}_{\epsilon} \rightarrow [a,b]$  such that 
for all $y \in D^{1,2}_{\epsilon},$
\begin{eqnarray} \label{a2p}
v_{\epsilon}^{1,2} (y)  &\geq&  E_{y}( e^{-\int_{0}^{\tau_{\alpha}} 
f_{\epsilon}(B_{t}) \,dt} v_{\epsilon}^{1,2}(B_{\tau_\alpha}) )+   E_{y}( 
\int_{0}^{\tau_{\alpha}} g_{\epsilon}(B_{t})v_{\epsilon}^{1}(B_{t}) 
v_{\epsilon}^{2}(B_{t})) \,dt    \\
\label{a2p1} v_{\epsilon}^{1,2} (y) &\leq&   E_{y}( 
e^{-\int_{0}^{\tau_{\alpha}} f_{\epsilon}(B_{t}) \,dt} 
v_{\epsilon}^{1,2}(B_{\tau_\alpha}) )  +E_{y}( \int_{0}^{\tau_{\alpha}} 
h_{\epsilon}(B_{t}) v_{\epsilon}^{1}(B_{t}) v_{\epsilon}^{2}(B_{t})) \,dt,\\
\label{a2l}
\lim_{\epsilon \rightarrow 0} g_{\epsilon}(y) &= &\lim_{\epsilon 
\rightarrow 0} h_{\epsilon}(y) = \psi^{(2)}(0)
\end{eqnarray}

\end{lemma}
{\bf Proof of Lemma:} Observe that,
\begin{eqnarray*}
\f{1}{2} \Delta v_{\epsilon}^{1,2}
&=&\sum_{\substack{ \emptyset \not =  B \subset \{1,2\}
}} (-1)^{ | B |  + 1 } \f{1}{2} \Delta u_{\epsilon}^B \\
&=&\sum_{\substack{\emptyset \not = B \subset \{1,2\}
}} (-1)^{| B | + 1 } \psi(u_{\epsilon}^B)\\
&=& \psi(v^{1}) + \psi(v_{\epsilon}^{2}) - \psi(v_{\epsilon}^{1} + 
v_{\epsilon}^{2} - v_{\epsilon}^{1,2}) \\
&=& -[\psi(v_{\epsilon}^{1} + v_{\epsilon}^{2}) - \psi(v_{\epsilon}^{1})] 
+ [\psi(v_{\epsilon}^{2}) - \psi(0)] + [\psi(v_{\epsilon}^{1} + 
v_{\epsilon}^{2})-  \psi(v_{\epsilon}^{1} + v_{\epsilon}^{2} - 
v_{\epsilon}^{1,2}) ]
\end{eqnarray*}
When $n=1$, (A2) implies that $\psi$ is twice continuously differentiable. Using Taylor's 
theorem on $\psi$, we have
\begin{eqnarray*}
\f{1}{2} \Delta v_{\epsilon}^{1,2}   &=& -[\psi^{(1)}(v_{\epsilon}^{1}) 
v_{\epsilon}^{2} + \psi^{(2)}(\alpha_{1}) 
\frac{(v_{\epsilon}^{2})^{2}}{2}] + [\psi^{(1)}(0) v_{\epsilon}^{2} + 
\psi^{(2)}(\alpha_{2}) \frac{(v_{\epsilon}^{2})^{2}}{2}] + 
[\psi^{(1)}(\alpha_{3}) v_{\epsilon}^{1,2}] \\
&=&  - [\psi^{(1)}(v_{\epsilon}^{1})  - \psi^{(1)}(0)] v_{\epsilon}^{2}  - 
[ \psi^{(2)}(\alpha_{1}) - \psi^{(2)}(\alpha_{2}) ] 
\frac{(v_{\epsilon}^{2})^{2}}{2}  + \psi^{(1)}(\alpha_{3}) 
v_{\epsilon}^{1,2}  \\
&=& -\psi^{(2)}(\alpha_{4}) v_{\epsilon}^{1}v_{\epsilon}^{2}  - [ 
\psi^{(2)}(\alpha_{1}) - \psi^{(2)}(\alpha_{2}) ] 
\frac{(v_{\epsilon}^{2})^{2}}{2}  + \psi^{(1)}(\alpha_{3}) 
v_{\epsilon}^{1,2}
\end{eqnarray*}
where $\alpha_{i} : D^{1,2}_{\epsilon} \rightarrow [0, \infty)$ are 
measurable functions  such that $ v_{\epsilon}^{1} \leq \alpha_{1} \leq 
v^{2}_{\epsilon} + v^{1}_{\epsilon}, 0 \leq \alpha_{2} \leq 
v^{2}_{\epsilon}, v_{\epsilon}^{1} + v_{\epsilon}^{2} -v^{1,2}_{\epsilon} 
\leq \alpha_{3} \leq v^{2}_{\epsilon} + v^{1}_{\epsilon},$ and $ 0 \leq 
\alpha_{4} \leq v^{1}_{\epsilon}. $

Repeating the above calculation with the roles of $v_{\epsilon}^{1}$ and 
$v_{\epsilon}^{2}$ reversed, we would obtain
\begin{eqnarray*}
\f{1}{2} \Delta v_{\epsilon}^{1,2} &=& -\psi^{(2)}(\beta_{4}) 
v_{\epsilon}^{1}v_{\epsilon}^{2}  - [ \psi^{(2)}(\beta_{1}) - 
\psi^{(2)}(\beta_{2}) ] \frac{(v_{\epsilon}^{1})^{2}}{2}  + 
\psi^{(1)}(\alpha_{3}) v_{\epsilon}^{1,2},
\end{eqnarray*}
where $\beta_{i} : D^{1,2}_{\epsilon} \rightarrow [0, \infty)$ are 
measurable functions  such that $ v_{\epsilon}^{2} \leq \beta_{1} \leq 
v^{2}_{\epsilon} + v^{1}_{\epsilon}, 0 \leq \beta_{2} \leq 
v^{1}_{\epsilon},$ and  $ 0 \leq \beta_{4} \leq v^{2}_{\epsilon}. $

Using the Feynman-Kac formula, we have
\begin{eqnarray*}
v_{\epsilon}^{1,2} (y) &=& E_{y}(e^{-\int_0^{\tau_\alpha} 
\psi^{(1)}(\alpha_{3}(B_r))\,dr } v_{\epsilon}^{1,2}(B_{\tau_\alpha}) ) + 
E_{y}(\int_0^{\tau_\alpha} e^{-\int_0^{t} \psi^{(1)} 
(\alpha_{3}(B_r))\,dr } J_{\epsilon}(B_{t}) )\,dt,
\end{eqnarray*}
where $J_{\epsilon}: D \rightarrow \R$ is given by
\begin{eqnarray*}
J_{\epsilon}&=& \psi^{(2)}(\alpha_{4}) v_{\epsilon}^{1}v_{\epsilon}^{2} + 
\left(\psi^{(2)}(\alpha_{1})- \psi^{(2)}(\alpha_{2}) \right 
)\frac{(v_{\epsilon}^{2})^{2}}{2} \\
&=& \psi^{(2)}(\beta_{4}) v_{\epsilon}^{1}v_{\epsilon}^{2} +\left 
(\psi^{(2)}(\beta_{1})- \psi^{(2)}(\beta_{2}) \right) 
\frac{(v_{\epsilon}^{1})^{2}}{2}.
\end{eqnarray*}
Now, $\psi^{2}(\cdot)$ is a non-negative decreasing continuous function. 
If $v_{\epsilon}^{1} \leq v_{\epsilon}^{2}$ then $\beta_{2} \leq 
\beta_{1}$ and  
$$
\psi^{(2)}(\beta_{4}) v_{\epsilon}^{1}v_{\epsilon}^{2}
\ge
J_{\epsilon}
\ge 
\left(\psi^{(2)}(\beta_{4})  +\frac{ 
\psi^{(2)}(\beta_{1})- \psi^{(2)}(\beta_{2})}{2} \right) 
v_{\epsilon}^{1}v_{\epsilon}^{2}.
$$
Likewise, if  $v_{\epsilon}^{2} \leq v_{\epsilon}^{1}$ then 
$\alpha_{2} \leq \alpha_{1}$
and 
$$
\psi^{(2)}(\alpha_{4}) v_{\epsilon}^{1}v_{\epsilon}^{2}
\ge
J_{\epsilon}
\ge 
\left(\psi^{(2)}(\alpha_{4})  +\frac{ 
\psi^{(2)}(\alpha_{1})- \psi^{(2)}(\alpha_{2})}{2} \right) 
v_{\epsilon}^{1}v_{\epsilon}^{2}.
$$
At any given point one of the above happens. 
Therefore (\ref{a2p}) and (\ref{a2p1}) hold with
\begin{eqnarray*}
f_{\epsilon} &=& \psi^{(1)}(\alpha_{3}),\\
g_{\epsilon} &=&  \min( \psi^{(2)}(\alpha_{4})   + 
\frac{\psi^{(2)}(\alpha_{1}) -\psi^{(2)}(\alpha_{2})}{2}, 
\psi^{(2)}(\beta_{4}) + 
\frac{\psi^{(2)}(\beta_{1})-\psi^{(2)}(\beta_{2})}{2}) \\
&& \mbox{ and, } \\
h_{\epsilon} &=&  \max( \psi^{(2)}(\alpha_{4}) , \psi^{(2)}(\beta_{4}) ). 
\\
\end{eqnarray*}
Note that for a given $\epsilon$, $f_{\epsilon} , g_{\epsilon}$ and 
$h_{\epsilon}$ are bounded functions. Also for all $i =1,2,3,4$, 
$\alpha_{i} \rightarrow 0$ as $\epsilon \rightarrow 0$ and for all $i 
=1,2,3$, $\beta_{i} \rightarrow 0$ as $\epsilon \rightarrow 0.$ So 
(\ref{a2l}) holds. \qed

The Lemma below, is the equivalent of  Lemma \ref{bound}.
\begin{lemma} Let $D$ be a bounded $C^{2}$ domain in dimension $d \geq 4$. Assume (A2) and that $a_1=0$. For $x \in D,   A \subset N$,   $\exists \,C < \infty, 
\epsilon_0 >0,$  such that $\forall \epsilon < \epsilon_0$ and $y \in 
D^A_{8 \epsilon}$
\begin{equation} \label{a2bound}
  \frac{ v^A_\epsilon(y) }{ \prod_{i \in A} v^{i}_\epsilon (x) }  \leq 
C\sum_{i \in A} K^D_x(y,z_i).
   \end{equation}
\end{lemma}

{\bf Proof of Lemma:} As before, one proceeds in two steps. The 
Step 1 proof, in Lemma \ref{bound}, follows verbatim. In Step 2,  use 
(\ref{a2p}) instead of (\ref{v^A}) to get
\begin{equation}
v_{\epsilon}^{1,2} (y) \leq  E_{y}( v_{\epsilon}^{1,2}(B_{\tau_\epsilon}) 
) +  E_{y}( \int_{0}^{\tau_{\alpha}} h_{\epsilon}(B_{t}) 
v_{\epsilon}^{1}(B_{t}) v_{\epsilon}^{2}(B_{t})\,dt).
\end{equation}
It is easy to see that the analysis used in obtaining (\ref{term1}), or 
from \cite{sv1}, will imply that  for $y \in D^{1,2}_{8 \epsilon},$
\begin{equation}\label{a2bound1}
\frac{E_{y}( v_{\epsilon}^{1,2}(B_{\tau_\epsilon}) ) }{v^{1}_{\epsilon}(x) 
v^{2}_{\epsilon}(x)} \leq c_{1}\epsilon^{2} \left ( K_{x}(y, z_{1}) + 
K_{x}(y, z_{2}) \right).
  \end{equation}
The induction hypothesis  and the fact that $h_{\epsilon}$ is bounded will 
imply
\begin{eqnarray}
\frac{E_{y}( \int_{0}^{\tau_{\alpha}} h_{\epsilon}(B_{t}) 
v_{\epsilon}^{1}(B_{t}) v_{\epsilon}^{2}(B_{t})\,dt)}{v^{1}_{\epsilon}(x) 
v^{2}_{\epsilon}(x)} \leq c_{2} \left ( K_{x}(y, z_{1}) +   K_{x}(y, 
z_{2}) \right).
\end{eqnarray}
So we have  proved the lemma. \qed

We are now ready to  state and prove  Theorem \ref{h1h2} assuming (A2) and 
not (A1), in the case $n=2$.

\begin{theorem} \label{2h1h2}  Let $n =2$, and let $D$ be a bounded $C^{2}$ domain 
in dimension $d \geq 4$. Assume (A2) and that $a_1=0$. Then for $x,y \in D$,
\begin{eqnarray}
\lim_{\epsilon \rightarrow 0 } \frac{ v^i_\epsilon(y) }{ v^{i}_\epsilon 
(x) }&= & K_x^D(y,z_i) \mbox{ for } i = 1,2  \label{a2l1}\\
\lim_{\epsilon \rightarrow 0 } \frac{ v^A_\epsilon(y) }{ \prod_{i \in A} 
v^{i}_\epsilon (x) }&= &   \psi^{(2)}(0) U\left(K^{D}_x(\cdot, z_{1}) 
K^{D}_x(\cdot, z_{2})\right )(y),\label{a2l2}
\end{eqnarray}
\end{theorem}

{\bf Proof : }  The proof of (\ref{a2l1}) is as in  Step 1 of the proof of 
Theorem \ref{h1h2}.  It essentially follows verbatim. Lemma \ref{unipr} 
was used  but that does not require (A1) or (A2).  The proof of 
(\ref{a2l2}) follows as in the induction step of the proof of Theorem 
\ref{h1h2}.  The ingredients for identifying the limit and application of 
dominated convergence are available immediately from 
(\ref{a2p}), (\ref{a2p1}), (\ref{a2l}),  (\ref{a2bound1}), (\ref{a2l2}) 
and (\ref{a2bound}) respectively.    \qed

\begin{remark}
From the above it seems likely that the same idea could work for $n \geq 
3.$ To do so would require a suitable appeal to Taylor's Theorem (as in  the 
proof of  (\ref{a2p})). We have not succeeded in carrying this out, and so have
presented the proof given earlier, under the
stronger assumption (A1).
\end{remark}

\subsection{Branching backbone for the limiting process, $\psi$ analytic}\label{limitBackbone}

The analysis of Section \ref{BrParticleDescrip} can also be carried out for the limiting conditioned process obtained above. We will simply state the conclusions one obtains, without repeating the derivation. 

Recall that 
\begin{equation}\label{MNk}
M^N_k=\sum_{\sigma\in\mathcal{P}(N)}\prod_{C\in\sigma}\langle X^k,K^C_x\rangle,
\end{equation}
where the $K^C_x$ are given in Theorem \ref{h1h2}. For $\Phi_k\in\mathcal{F}_k$, we let 
$$
\M^N_y (\Phi_k)=\frac{1}{K^N_x(y)}\N_y(\Phi_kM^N_k)
$$
be the limiting measure arising in Theorem \ref{limitingXtransform}. Then the following statements hold, under the conditions of that result. 
\begin{itemize}
\item $M^N_k$ is a martingale with respect to $\mathcal{F}_k$, so the associated Girsanov transform $\M^N_k$ defines a consistent probability measure on $\lor_k\mathcal{F}_k$. In other words, in Dynkin's terminology, $M^N_k$ is an $X$-harmonic function of $X$, and $\M^N_y$ is its $X$-transform. 
\item This probability can equivalently be described in terms of a branching backbone throwing off mass, that is, as a superprocess with immigration along a random set obtained from a branching tree of particles. 
\item The backbone starts with a single particle, located initially at $y$. It performs a $K^N_x$-transform of Brownian motion, which dies somewhere in the interior of $D$. Say it dies at $\hat y$.
\item A random partition $\Sigma$ of $N$ is chosen, so given $\hat y$, the probability that $\Sigma=\sigma$ is
$$
\frac{1}{V^N(\hat y)}b(|\sigma|,0,\hat y)\prod_{A\in\sigma}K^A_x(\hat y),
$$
where $\sigma\in\mathcal{P}(N)$, $|\sigma|\ge 2$. 
Here 
$$
V^N(\hat y)=\sum_{\sigma\in\mathcal{P}(N), |\sigma|\ge 2}b(|\sigma|,0,\hat y)\prod_{A\in\sigma}K^A_x(\hat y).
$$
\item For each $A\in\Sigma$, a particle is born at $\hat y$ which proceeds to carry out a $K^A_x$-transform of Brownian motion. If $A=\{z_i\}$, this particle survives to exit $D$ at $z_i$. If $|A|\ge 2$ then the particle dies in the interior of $D$ and is replaced by a random number of children, labeled by a random partition of $A$ in the manner described above. This process repeats until all partitions consist of singletons. This process produces a branching tree of particles, with precisely $n$ leaves, corresponding to particles exiting $D$ at the $n$ points of $N$. 
\item Mass is created/immigrated at points of $D$ where the backbone branches. Given that $j$ particles are born because of a branch at $\hat y$, the mass created is a random variable $R\ge 0$ whose conditional law $\mu^{j}(dr)$ is 
$$
\mu^{j}=
\begin{cases}
\frac{r^j\,\pi(dr)}{\int_0^\infty r^j\,\pi(dr)}, &j\ge 3\\
\frac{r^2\,\pi(dr)}{2a_2+\int_0^\infty r^2\,\pi(dr)}, &j=2.
\end{cases}
$$
\item Mass is also created continuously along the backbone according to a L\'evy process, with L\'evy exponent
$$
\eta(\lambda)=\psi'(\lambda)-\psi'(0)=2a_2\lambda+\int_0^\infty r(1-e^{-\lambda r})\, \pi(dr). 
$$
\item Once created, the mass evolves as the (unconditioned) $\psi$-super Brownian motion. In other words, with excursion law $\N_\cdot$. 
\end{itemize}

As remarked earlier, note that unlike \eqref{martingale}, there is no factorial factor in \eqref{MNk}. This is simply because of the indexing scheme for the backbones. In Section \ref{BrParticleDescrip} we ordered the $n$ points and then labeled them at random. While now we label using the natural order given by the partition. In that sense, the current indexing scheme parallels that of \cite{sv1}.

The argument is an induction, based on the Palm formula (Lemma \ref{palmformula}), as in Theorem \ref{theorem}. But note that each of the local characteristics of the backbone and mass evolution given above are consistent with sending $\epsilon\to0$ in the characteristics of Section \ref{BrParticleDescrip}

\subsection{Explosion of Mass as $\epsilon \rightarrow 0$}\label{explosion}

The explosion effect we wish to understand is already present when $n=2$, so we will focus mainly on that case. To reiterate the representation for analytic $\psi$ in this special case, write $K^i_x=K^D_x(\cdot,z_i)$. Then the process of interest is the Girsanov transform by the martingale
$$
M^2_k=\langle X^k,K^1_x\rangle\langle X^k,K^2_x\rangle+m_2\langle X^k,U(K^1_xK^2_x)\rangle,
$$
where 
$$
m_2=-\frac{\psi''(0)}{2}=a_2+\frac12\int_0^\infty r^2\, \pi(dr).
$$
The backbone has a single branch, and follows a $U(K^1_xK^2_x)$-transform till it dies, whereupon two particles are born, one doing a $K^1_x$-transform, and the other doing a $K^2_x$-transform. Mass is created at the branch point, according to the law $\mu^2$, and continuously along the backbone, according to the L\'evy exponent $\eta$. The mass then evolves as an unconditioned super-Brownian motion. 

Consider the stable branching function
$$
\psi_\beta(\lambda)=\int_0^\infty r^{-(\beta+2)}(e^{-\lambda r}-1 + \lambda r)\,dr=c_\beta \lambda^{1+\beta}
$$
(for $c_\beta$ chosen appropriately), which satisfies \eqref{condition} for $0<\beta<1$. (A1) fails for $\psi_\beta$, but it does apply to 
$$
\psi_\beta^\gamma(\lambda)=\int_0^\infty r^{-(\beta+2)}e^{-\gamma r}(e^{-\lambda r}-1 + \lambda r)\,dr,
$$
and clearly $\psi_\beta^\gamma\to\psi_\beta$ as $\gamma\downarrow 0$. Thus our construction applies to $\psi^\gamma_\beta$, giving an $X$-transform $\M^{2,\gamma}_y$ with density 
$$
\tilde M^\gamma_k=\frac{1}{K^{\{1,2\}}_x(y)}M^{2,\gamma}_k
=\frac{\langle X^k,K^1_x\rangle\langle X^k,K^2_x\rangle+m_2^\gamma\langle X^k,U(K^1_xK^2_x)\rangle}{m_2^\gamma U(K^1_xK^2_x)(y)},
$$
where $m_2^\gamma=\frac12\int_0^\infty r^{-\beta}e^{-\gamma r}\,dr\to\infty$ as $\gamma\downarrow 0$. Thus 
$$
\tilde M^{\gamma}_k\to \tilde M_k=\frac{1}{U(K^1_xK^2_x)(y)}{\langle X^k,U(K^1_xK^2_x)\rangle},
$$
which does not represent a valid Girsanov transform, because $\tilde M_k$ is not a martingale in $k$ under $\N_y$. It is not a surprise that something goes wrong here, because expectations of terms like $\langle X^k,K^1_x\rangle\langle X^k,K^2_x\rangle$ should blow up as $\gamma\downarrow 0$ -- after all, stable random variables don't have finite second moments. Our original motivation for carrying out the analysis of this paper was understanding precisely what goes wrong in the stable case. In other words, of understanding how the singularity arises.

The problem is not the backbone, since the description of the backbone does not even depend on $\gamma$. Nor is the problem the continuous mass creation or the subsequent evolution of mass, since those approach the corresponding mechanisms for the $\psi_\beta$-super Brownian motion. The problem is precisely the creation mechanism $\mu^{2,\gamma}(dr)$ of mass at the branch point, since its density is proportional to $r^{-\beta}e^{-\gamma r}$, which fails to be tight when $\gamma\downarrow 0$. In other words, as $\gamma\downarrow 0$, the mass born at the branch point blows up. 

While this heuristic analysis is sufficient to explain the singularity, one can use it to rigorously explain the change of measure by the $\tilde M_k$. Let $\zeta$ be the time the original particle in the backbone dies. Let $\tau_k$ be the lesser of the time it dies and the time it exits $D_k$. 

\begin{theorem}\label{supermart}
$\tilde M_k$ is a supermartingale. 
Let $\phi\ge 0$. Then for $y\in D_k$,
$$
\N_y(\tilde M_k\exp-\langle X^k,\phi\rangle)=\lim_{\gamma\downarrow 0} \M^{2,\gamma}_y(\exp-\langle X^k,\phi\rangle, \tau_k<\zeta)
$$
\end{theorem}

\begin{proof}
Just follow the argument of Theorem \ref{theorem}. In the notation of that result, the event $\{\tau_k<\zeta\}$ is exactly the same as $\{\Upsilon^k\approx (\{1,2\})\}$, and the probabilities of such events entered into
the proof of Theorem \ref{theorem}.
\end{proof}

The interpretation of the supermartingale property is that we lose absolute continuity between the two measures in the limit, with $\M_y^{2,\gamma}$ increasingly concentrating mass near a set of $\N_y$-measure 0. Namely the set where mass explodes. 

A similar analysis can be carried out in the case $n>2$. One shows by induction that as $\gamma\to 0$, one has $b(j,0,y)=(-1)^j\psi^\gamma_\beta(0)\sim c^1_{\beta,j}\gamma^{1+\beta-j}$, and $K^A_x\sim c^{|A|}_\beta \gamma^{1+\beta-|A|}U(\prod_{k\in A}K^k_x)$ for $|A|>2$. It follows that asymptotically there is a single branch point with $|A|$ branches, at which the backbone changes from a single $K^A_x$ transform, to $|A|$ transforms by $K^k_x$, $k\in A$. Moreover, the mass created at this branch point blows up as $\gamma\to 0$. For example, in the case $n=3$, $\sigma=\{\{1\},\{2\},\{3\}\}$ has probability proportional to $\gamma^{\beta-2}$, while $\sigma=\{\{1,2\},3\}$ has probability proportional to $\gamma^{\beta-1}\gamma^{\beta-1}=\gamma^{2\beta-2}$ which is of lower order of magnitude.

\subsection{Other orders of limits}\label{otherlimits}

Still in the analytic case, with $n=2$, the same arguments would have handled a slightly more general conditioning. Namely, condition on the exit measure charging both $\Delta^{z_1}_{\epsilon_1}$ and $\Delta^{z_2}_{\epsilon_2}$, where $\epsilon_1>0$ and $\epsilon_2>0$. First send $\epsilon_1\to 0$ and then $\epsilon_2\to 0$. It is not hard to check that this gives precisely the same limiting object as before. So if we apply this procedure to the $\psi_\beta^\gamma$-super Brownian motion, and then let $\gamma\downarrow 0$, the mass should still blow up, but the backbone should take the form of a $U(K^1_xK^2_x)$-transform, splitting into a $K^1_x$-transform and a $K^2_x$-transform. 

With this modification, one should be able to treat a different approach to the stable case than given above. Namely, start out with the actual $\psi_\beta$-super Brownian motion. Condition its exit measure to charge $\Delta^{z_1}_{\epsilon_1}$ and $\Delta^{z_2}_{\epsilon_2}$. Then send $\epsilon_1\to 0$ followed by $\epsilon_2\to 0$. The following informal calculation suggests how the backbone should behave in the limit.

For $i=1,2$ set
\begin{align*}
&u^i=v^i=u^i_{\epsilon_i}=v^i_{\epsilon_i}=\N_\cdot( X^D(\Delta^{z_i}_{\epsilon_i})>0)\\
&u^{12}=u^{12}_{\epsilon_1,\epsilon_2}=\N_\cdot( X^D(\Delta^{z_1}_{\epsilon_1}\cup\Delta^{z_2}_{\epsilon_2})>0)\\
&v^{12}=v^{12}_{\epsilon_1,\epsilon_2}=u^1+u^2-u^{12}=
\N_\cdot( X^D(\Delta^{z_1}_{\epsilon_1})>0, X^D(\Delta^{z_2}_{\epsilon_2})>0).
\end{align*}
Then 
\begin{align*}
\frac12\Delta v^{12}&=\frac12\Delta [u^1+ u^2-u^{12}]
=\psi_\beta(u^1)+\psi_\beta(u^2)-\psi_\beta(u^{12})\\
&=\psi_\beta(v^1)+\psi_\beta(v^2)-\psi_\beta(v^1+v^2-v^{12})\\
&=\psi_\beta(v^1)+[\psi_\beta(v^2)-\psi_\beta(v^1+v^2)]+[\psi_\beta(v^1+v^2)-\psi_\beta(v^1+v^2-v^{12})]\\
&\sim c_\beta\Big([v^1]^{1+\beta}-[1+\beta][v^2]^\beta v^1+[1+\beta][v^1+v^2]^\beta v^{12}\Big).
\end{align*}
Fix $x$ and rescale, letting $\bar v^1=v^1(\cdot)/v^1(x)$, $\bar v^2=v^2(\cdot)/v^2(x)$, 
$\bar v^{12}=v^{12}(\cdot)/v^1(x)v^2(x)^\beta$. If all these quantities remain bounded, then sending $\epsilon_1\to 0$ should give
$$
\frac12\Delta\bar v^{12}\sim(1+\beta)c_\beta\Big([v^2]^\beta\bar v^{12}-[\bar v^2]^\beta\bar v^1\Big).
$$
So sending $\epsilon_1\to 0$ and then $\epsilon_2\to 0$ should give
$$
\frac12\Delta\bar v^{12}\sim-(1+\beta)c_\beta[\bar v^2]^\beta\bar v^1.
$$
Though we again expect this conditioning to have a mass which blows up in the limit, the above suggests that the backbones should still converge weakly. But this time the limit should be a $U([K^2_x]^\beta K^1_x)$-transform, splitting into a $K^1_x$-transform and a $K^2_x$-transform. 

Call the first part of the backbone the {\sl trunk}. Taking limits in the other order should produce a trunk that is a $U([K^1_x]^\beta K^2_x)$-transform. So taking limits in which both $\epsilon_1\to 0$ and $\epsilon_2\to 0$ in a coordinated way should give trunks that are transforms by 
$$
U\Big([K^1_xK^2_x]^\beta [\theta K^1_x+(1-\theta)K^2_x]^{1-\beta}\Big),
$$ for arbitrary $0\le\theta\le 1$. 

We will not try to make the above informal argument rigorous. But we originally found it puzzling. Not so much because uniqueness of limits breaks down in the case of the $\psi_\beta$-super Brownian motion. But rather because the $U(K^1_xK^2_x)$ trunk we obtained earlier does not appear among the possible limits when taken this other way. This suggests that there should be a way of capturing a broader class of limits, encompassing both types obtained above. Or at least, of obtaining both types of limits by a common procedure. Carrying this out rigorously, in the current context, seems more work than it is perhaps worth, and we do not claim to have done so in full detail. But in the next section, we will find a simpler setting, in which this can be done more more easily. 

\medskip

\subsection{A class of martingales, with $n=2$} \label{spectrum}

In this section, we focus on a technically simpler collection of martingales, exhibiting some of the same behaviour as found above. We  only present the heuristic idea and do not present detailed proofs in this section. The heuristic development can be made rigorous but we have chosen not to include this, as our aim in this section is to illustrate how a spectrum of limits can be obtained via various limiting procedures..

Let $\psi$ satisfy (A1). Let $u, f, g, v>0$ be bounded solutions on $D$ to the following equations:
\begin{equation}\label{martclass}
\left\{
\begin{aligned}
\frac12\Delta u &=\psi(u)\\
\frac12\Delta f -\psi'(u)f&=0\\
\frac12\Delta g -\psi'(u)g&=0\\
\frac12\Delta v -\psi'(u)v&=-\psi''(u)fg
\end{aligned}
\right.
\end{equation}
as well as $v=0$ on $\partial D$ (so $v$ is an $L_{\psi'\circ u}$-potential). Call $\mathcal{M}_2$ the set of $(u,f,g,v)$ so obtained. 

Under $\M_x$, for every $(u,f,g,v)\in\mathcal{M}_2$ define:
$$
M_k(u,f,g,v)=e^{-\langle X^k,u\rangle}\Big[ \langle X^k,v\rangle + \langle X^k,f\rangle\langle X^k,g\rangle\Big].
$$
It can be verified that  $M_k(u,f,g,v)$ is a $\mathcal{F}_k$ martingale. Moreover, we can realize the Girsanov transform 
$$
\M_y^{(u,f,g,v)}(\Phi_k)=\frac{1}{v(y)}\N_y(\Phi_kM_k^{(u,f,g,v)})
$$
by such martingales, as follows: Let $(u,f,g,v)\in\mathcal{M}_2$. Then $\M_y^{(u,f,g,v)}$ can be described as follows: Start a $v$-transform of the $L_{\psi'\circ u}$ process till it dies in $D$. At this point, start two paths that run to $\partial D$, respectively an $f$-transform and a $g$-transform of the $L_{\psi'\circ u}$ process.  Mass is created at the branch point, according to the law $\mu^2$, and continuously along the backbone, according to the L\'evy exponent $\eta$. The mass then evolves as an unconditioned super-Brownian motion. 

In this simplified context, the analogue to the question considered in previous sections is the following: Let $(u_n,f_n,g_n,v_n)\in\mathcal{M}_2$, with each term $\to 0$. Find the limit points, either of $\M_y^{(u_n,f_n,g_n,v_n)}$, or of the backbone. 

When $\psi$ satisfies (A1), there is a systematic answer to this question. For a fixed $x$, we  renormalize by constants $a_n$ and $b_n$, so that
$$
\bar f_n=\frac{f_n}{a_n}, \quad
\mbox{ and } \bar g_n=\frac{g_n}{b_n}, \quad
$$ 
converge to non-zero values. Then $\bar f_n$, $\bar g_n$, and $\bar v_n=\frac{v_n}{a_nb_n}$ all have non-zero limits $f, g, v$ which satisfy 
$$
\frac12 \Delta f = \frac12\Delta g=0, \quad v=\psi''(0)U(fg). 
$$
Moreover $\M_y^{(u_n,f_n,g_n,v_n)}$ converges in the weak topology to $\M_y^{(0,f,g,v))}$. 

In the case of $\psi_\beta$, the actual Girsanov transforms $\M_y^{(u_n,f_n,g_n,v_n)}$ degenerate as before, because mass blows up. We look at the limits of the backbones instead. 

Let $a_n, b_n, c_n$ all be sequences of positive reals, that all $\to 0$. Assume that $a_nb_n=o(c_n^{1-\beta})$. Let $\phi_f, \phi_g, \phi_u$ be smooth functions on $\partial D$, that are bounded away from 0.  Define $(u_n,f_n,g_n,v_n)$ to be the solutions to \eqref{martclass}, using boundary conditions $a_n\phi_f$ for $f_n$, $b_n\phi_g$ for $g_n$, and $c_n\phi_g$ for $u_n$. Set $d_n=\beta(1+\beta)a_n b_n c_n^{-(1-\beta)}$ and
$$
\bar u_n=\frac{u_n}{c_n}, \quad
\bar f_n=\frac{f_n}{a_n}, \quad
\bar g_n=\frac{g_n}{b_n}, \quad
\bar v_n=\frac{v_n}{d_n}.
$$
\eqref{martclass} gives that
$$
\frac12\Delta v_n-(1+\beta)u_n^\beta v_n=-\beta(1+\beta)u_n^{-(1-\beta)}a_nb_n\bar f_n\bar g_n,
$$
so 
$$
\frac12 \Delta \bar v_n-(1+\beta)u_n^\beta \bar v_n=-\bar u_n^{-(1-\beta)}\bar f_n\bar g_n.
$$
Now taking limits we have $\bar u_n\to u$, $\bar f_n\to f$, $\bar g_n\to g$, and $\bar v_n\to v$ where $u$, $f$, and $g$ are the harmonic functions with boundary values $\phi_u$, $\phi_f$, $\phi_g$, and $v=U(u^{-(1-\beta)}fg)$. The backbones converge weakly  to a $v$-transform branching into an $f$-transform and a $g$-transform. 

We conclude this section with two of examples analogous to the limits obtained in the previous sections.

\begin{example}\label{abc1}
Take $\phi_u=1$. The trunk is a $U(fg)$-transform, as in Section \ref{explosion}
\end{example}

\begin{example}\label{abc2}
Take $\phi_u=\phi_f$. The trunk is a $U(f^\beta g)$-transform, as in Section \ref{otherlimits}
\end{example}

In other words, this gives a general class of limiting objects, which encompasses examples analogous to  both types of limits obtained earlier. In that sense, the model of this section interpolates between these examples, and helps explain the variety of limits obtained.


\begin{thebibliography}{99}


\bibitem{al} Abraham, R. and Le Gall, J.F. (1994) {\sl Sur la mesure de sortie du super mouvement brownien.} Probabability  Theory and Related Fields 99, 251-275.

\bibitem{a} Az\`ema, J. (1979) {\sl Th\'eorie g\'en\'erale des processus et retournement du temps. } Ann. Sci. \'Ecole Norm. Sup. 6, 459-519.

\bibitem{bbb} Banuelos, R., Bass, R., and Burdzy, K. (1991) {\sl H\"older
domains and the boundary Harnack principle.} Duke Math. J. 64, no. 1, 195-200.

\bibitem{cz} Chung, K. L. and Zhao, Z.  (1995) {\sl From Brownian motion to  Schr\"odinger's equation.} Springer-Verlag, Berlin.

\bibitem{cfz} Cranston, M., Fabes, E., and Zhao, Z.  (1988) {\sl Conditional gauge  and potential theory for the Schr\"odinger operator.} Trans. Amer. Math. Soc. 3307, 171-194.
 
\bibitem{dp91} Dawson, D. A. and Perkins, E. A. (1991) {\sl Historical
Processes.} Memoirs of the American Mathematical Society, No. 454.

\bibitem{daw} Dawson, D. A. (1991)  {\sl Measure-valued Markov Processes. } { Ecole d'Et\'e de Probabilit\'es de Saint-Flour XXI}, Lecture Notes in Mathematics 1541, Springer-Verlag New York.

\bibitem{dy0}  Dynkin, E.B. (1991) {\sl  An Introduction to Branching Measure-Valued Processes.} CRM Monographs series 6: Amer. Math. Soc., Providence, R. I.

\bibitem{dy1} Dynkin, E. B. (1991) {\sl A Probabilistic Approach to One Class
of Nonlinear Differential Equations.} Probability Theory and Related Fields 90, 1-36

\bibitem{dy02}Dynkin, E. B. (2002)  {\sl Diffusions, Superdiffusions and Partial Differential Equations,} Amer. Math. Soc., Providence, R.I.

\bibitem{dy041}Dynkin, E. B. (2004)  {\sl Harmonic functions and exit boundary of superdiffusion,}  Journal of functional analysis 206, 33-68.

\bibitem{dy042}Dynkin, E. B. (2004) {\sl Superdiffusions and positive solutions of nonlinear partial differential equations.} University Lecture Series {\bf 34}, Amer. Math. Soc., Providence, R. I.

\bibitem{enpi} Englander, J. and Pinsky, R. (1999) {\sl On the construction and support properties of measure-valued diffusions on   $D\subseteq{\bf R}^d$ with spatially dependent branching. } Ann. Probab. 27,  no. 2, 684-730. 

\bibitem{etheridge}Etheridge, A.M. (1993) {\sl Conditioned superprocesses and a semilinear heat equation}, in ``Seminar on Stochastic Processes 1992'', Birkh\"auser, Boston, 91Ð99.
   
\bibitem{ew}Etheridge, A. M. and Williams, D. R. E. (2003) {\sl  A decomposition of the $(1+\beta)$-superprocess conditioned on survival.}  Proc. Roy. Soc. Edinburgh Sect. A  133, 829--847.

\bibitem{ev} Evans, S.N.  (1993) {\sl Two representations of a conditioned superprocess.} Proc. Roy. Soc. Edinburgh Sect. A 123, 959-971.


\bibitem{ep}  Evans, S.N. and  Perkins, E.A. (1990) {\sl Measure-valued Markov branching processes conditioned on non-extinction.} Israel J. Math, 71: 329-337.

\bibitem{g} Grey, D. R. (1974) {\sl Asymptotic behaviour of continuous time, continuous state-space branching processes.} J. Appl. Probab. 11, 669-677

\bibitem{klmr} Kyprianou, A.E., Liu, R.-L., Murillo-Salas, A., and  Ren, Y.-X. (2010) {\sl Supercritical super-Brownian motion with a general branching
mechanism and travelling waves.} Preprint.

\bibitem{legall1}  Le Gall, J.F. (1994)  {\sl Hitting probabilities and potential theory for the Brownian path-valued process.}  Ann. Inst. Fourier (Grenoble) 44, 277-?306.

\bibitem{legall2}  Le Gall, J.F. (1996) {\sl  Superprocesses, Brownian snakes and partial differential equations. } Prepublication,  Lab. de Probabilites, Univ. Paris VI 337.

\bibitem{mm} Meir, A. and Moon, J.W. (1978)  {\sl On the altitude of nodes in random trees.} Canad. J. Math. 30, no. 5, 997-1015.

\bibitem{moras} Moras, M. (2009) {\sl Conditioned super Brownian motion in Denjoy domains and strips.} Ph. D. Thesis, York University. 

\bibitem{ov}  Overbeck,  L. (1993)  {\sl Conditioned super-Brownian motion.}  Probab. Theory Relat. Fields, 96, 545-570.

\bibitem{ov2}  Overbeck,  L. (1994) {\sl Pathwise construction of additive $H$-transforms of super-Brownian motion.} Probab. Theory Relat. Fields 100, 429Ð437.

\bibitem{p} Perkins, E. A. (2002) {\sl Dawson-Watanabe superprocesses and
   measure-valued diffusions. }  Lecture Notes in Mathematics, 1781. Springer-Verlag, Berlin, (125--329).

\bibitem{rr1} Roelly-Coppoletta, S. and Rouault, A. (1989) {\sl Processus de Dawson-Watanabe conditionn\'e par  le futur lointain.}  C.R. Acad. Sci. Paris, Ser. 1 309, 867-872. 

\bibitem{ssezer} Salisbury, T.S. and Sezer, A.D. (2010) {\sl Conditioning super-Brownian motion on its boundary statistics; fragmentation and a class of weakly extreme $X$-harmonic functions.} Preprint. 

\bibitem{sv1} Salisbury, T. S. and Verzani, J. (1999) {\sl On the conditioned
exit measures of super Brownian motion.} Probability Theory and Related Fields 115, 237-285.

\bibitem{sv2} Salisbury, T. S. and Verzani, J. (2000) {\sl Non-degenerate
conditionings of the exit measures of super-Brownian motion.}  Stochastic Processes and Their Applications 87, 25-52.

\bibitem{serlet} Serlet, L. (1996) {\sl The occupation measure of super-Brownian motion conditioned on nonextinction.}  J. Theoret. Probab. 9, 561-578.

\bibitem{ver} V\'eron, L. (1996) {\sl Singularities of solutions of second order quasilinear equations.} Pitman Research Notes in Mathematical Series 353, Addison Wesley Longman Ltd, U.K.

\bibitem{verzani} Verzani, J. (2008) (\sl A class of extreme $X$-harmonic functions.) Commun. Stoch. Anal. 2, 335-355.

\end{thebibliography}
\end{document}